\documentclass[11pt,twoside, leqno]{article}

\usepackage{amssymb,amsmath,amsfonts,amsthm,color,mathrsfs}
\usepackage[Symbol]{upgreek}
\usepackage{txfonts}
\usepackage[nottoc,notlot,notlof]{tocbibind}
\usepackage[active]{srcltx}
\usepackage{citeref}
\usepackage{hyperref}
\usepackage[dvipsnames, svgnames, x11names]{xcolor}

\usepackage{graphicx}
\usepackage{epsfig}

\usepackage{geometry}
\usepackage{graphicx}

\usepackage{subfig}
\usepackage{caption}

\usepackage{txfonts}
\allowdisplaybreaks
\def\rr{{\mathbb R}}
\def\rn{{{\rr}^n}}

\def\cn{{\mathbb N}}

\def\dist{{\mathop\mathrm {\,dist\,}}}

\newtheorem{thm}{Theorem}[section]
\newtheorem{lem}[thm]{Lemma}
\newtheorem{prop}[thm]{Proposition}
\newtheorem{cor}[thm]{Corollary}

\newtheorem{quest}[thm]{Question}

\newtheorem{defn}[thm]{Definition}
\newtheorem{rem}[thm]{Remark}

\numberwithin{equation}{section}

\begin{document}
\arraycolsep=1pt
\author{Renjin Jiang \& Fanghua Lin}
\title{{\bf P\'olya's conjecture up to $\epsilon$-loss and quantitative estimates for the remainder of Weyl's law}
 \footnotetext{\hspace{-0.35cm} 2020 {\it Mathematics
Subject Classification}. Primary  35P15; Secondary  35P20, 42B37.
\endgraf{
{\it Key words and phrases: Dirichlet eigenvalue, Weyl law, P\'olya conjecture}
\endgraf}}
\date{\today}}
\maketitle

\begin{center}
\begin{minipage}{11.5cm}\small
{\noindent{\bf Abstract}. Let $\Omega\subset\mathbb{R}^n$ be a bounded Lipschitz domain.
For any $\epsilon\in (0,1)$ we show that for any Dirichlet eigenvalue $\lambda_k(\Omega)>\Lambda(\epsilon,\Omega)$, it holds
\begin{align*}
k&\le  (1+\epsilon)\frac{|\Omega|\omega(n)}{(2\pi)^n}\lambda_k(\Omega)^{n/2},
\end{align*}
where $\Lambda(\epsilon,\Omega)$ is given explicitly. This reduces the $\epsilon$-loss version of P\'olya's conjecture to a computational problem.
This estimate is based on quantitative estimates on the remainder of the Weyl law with explicit constants,
which we give a new proof without using Neumann eigenvalues.
Our arguments in deriving such uniform estimates yield also, in all dimensions $n\ge 2$, classes of domains that may even have rather irregular shapes or boundaries but satisfy P\'olya's conjecture.
Another key observation is that on strip-tiling domains (and therefore any triangles for instance) one actually has better eigenvalue estimates than P\'olya conjectured.
}\end{minipage}
\end{center}
\vspace{0.2cm}
\tableofcontents

\section{Introduction}
\subsection{Weyl's law and P\'olya's conjecture}
\hskip\parindent
Given an open bounded domain $\Omega$ in $\rn$, $n\ge 2$. Let
$$0<\lambda_1(\Omega)\le \lambda_2(\Omega)\le \lambda_3(\Omega)\le \cdots$$
be eigenvalues to the Dirichlet Laplacian on $\Omega$.
We shall simply write $\lambda_k$
instead of $\lambda_k(\Omega)$ when there is no confusion.

In 1911, Weyl \cite{Weyl11} proved the famous Weyl's law, which  states  that for any bounded open domain $\Omega\subset\rn$, $n\ge 2$, and $\lambda\to\infty$,
\begin{align}
\mathcal{N}^D_\Omega(\lambda)=\frac{|\Omega|\omega(n)}{(2\pi)^n}\lambda^{n/2}+o(\lambda^{n/2}),
\end{align}
where $\omega(n)$ denotes the volume of unit ball in $\rn$, $\mathcal{N}^D_{\Omega}(\lambda)$ is the counting function
$$\mathcal{N}^D_{\Omega}(\lambda):=\#\{\lambda_k(\Omega):\,\lambda_k(\Omega)<\lambda\}.$$
The Weyl law is also important in the longstanding question ``can one hear about the shape of a cavity",
see Lord Rayleigh \cite{Ray77}  and Kac \cite{Kac66}. Moreover, the Weyl law maybe related to the
prime number theorem,  via  Connes' trace formula, see Connes \cite{connes99},
Fathizadeh-Khalkhali \cite{FK13}  and the beamer by
Khalkhali \cite{Kha11}.

Let us denote the remainder of Weyl's law as
\begin{align}
R_\Omega(\lambda)= \mathcal{N}^D_{\Omega}(\lambda)-\frac{|\Omega|\omega(n)}{(2\pi)^n}\lambda^{n/2}.
\end{align}
Subsequent to Weyl's seminal work, Courant \cite{courant20},  Courant-Hilbert \cite{CH89}, H\"ormander \cite{Hormander68,Hormander3}, Seeley \cite{Seeley78,Seeley80}
and many others (see Ivrii \cite{Ivr16} for comprehensive description on related history) made contributions to improve the estimate on the remainder.

Weyl also  conjectured a sharper asymptotic behavior of Dirichlet eigenvalue states for domains with piecewise smooth boundary that
\begin{align}\label{conj-weyl}
\mathcal{N}^D_{\Omega}(\lambda)=\frac{|\Omega|\omega(n)}{(2\pi)^{n}} \lambda^{n/2}-\frac{|\partial\Omega|}{2^{n+1} \pi^{\frac{d-1}{2}}\Gamma(\frac{n+1}{2})}\lambda^{\frac{n-1}{2}}+o(\lambda^{\frac{n-1}2}),
\end{align}
here and in what follows,  for any set $\Omega$, $|\partial \Omega|$ denotes the surface area of it, $\mathscr{B}(\cdot,\cdot)$ the Beta function and $\Gamma(\cdot)$ the Gamma function, as usual.
The conjecture has been proved by Ivrii \cite{Ivr80} under the condition that the set of all periodic geodesic billiards in $\Omega$ has measure zero, which was proved by Safarov-Vassiliev \cite{SV97} for convex analytic domains
and polygons. {In a recent impressive work \cite{FL24b}, Frank and Larson established   sharper asymptotic formula
for Riesz means of eigenvalues for all exponent $\gamma>0$ (averaged version of \eqref{conj-weyl})  on general Lipschitz domains.}

However, the above mentioned results  for the remainder $R_\Omega(\lambda)$  state that certain asymptotic estimates valid only for large $\lambda$'s without quantifying the sizes of such $\lambda$'s,
and consequently they do not give  quantitative estimate on all eigenvalues.

It is rather natural to ask the following question.
\begin{quest}\label{question1}
Is it possible to give a uniform and quantitative estimate for the remainder $R_\Omega(\lambda)$ of Weyl's law?
\end{quest}
It was proved earlier by Netrusov and Safarov \cite{NS05} (see also \cite{FLW23}), that a quantitative bound  on rough domains up
to an extra logarithm term of the remainder of Weyl's law holds.
{Recently in \cite{FL24b}, Frank and Larson provided yet another uniform estimate of the remainder $R_\Omega(\lambda)$ on Lipschitz domains. }
\footnote[1]{After the first version of the present article was posted on arXiv,
R.  Frank and S. Larson has kindly informed us recent works  \cite{FLW23} and \cite{FL24b}, and X. He also informed us the work \cite{NS05}, in which several uniform estimates for the remainder were established. }

In 1954, P\'olya \cite{Polya54}  conjectured that, it should hold on arbitrary bounded open domain that
$$ k\le \frac{|\Omega|\omega(n)}{(2\pi)^n}\lambda_k^{n/2},\ \forall\, k\in\cn.$$
P\'olya himself \cite{Polya61} proved this conjecture for  tiling domains   in the plane.
His method extends to high dimensions, and so the conjecture is true for tiling domains in high dimensions too.

The P\'olya conjecture has been wide open since then. Other than tiling domains,
Laptev \cite{La97} verified  P\'olya's conjecture on product domains $\Omega_1\times\Omega_2\subset\rr^{n_1+n_2}$, $n_1\ge 2$,
$n_2\ge 1$, provided P\'olya's conjecture holds on $\Omega_1$. Filonov et al. \cite{FLPS23} and \cite{FLPS25}
verified, via number theoretic arguments,   P\'olya's conjecture on balls for all $n\ge 2$, and on annuli  for $n=2$.
For general domains $\Omega$, it was known that P\'olya's conjecture holds for the first two Dirichlet eigenvalues
see  \cite{Henrot06}  for the Rayleigh-Faber-Krahn inequality ($\lambda_1(\Omega)$) and   the Krahn-Szeg\"o inequality  ($\lambda_2(\Omega)$).
See also \cite{Henrot06} for the first nontrivial Neumann eigenvalue and recent exciting developments \cite{BuHe19,GNP09}
for the second nontrivial Neumann eigenvalue.

For general domains, Berezin \cite{Berezin72} and  Li-Yau \cite{LY83} independently proved for arbitrary bounded open domains that
\begin{equation}\label{ly-sum}
\frac{ |\Omega|^{2/n}\omega(n)^{2/n}}{(2\pi)^2}\sum_{i=1}^k\lambda_i \ge \frac{n}{n+2} k^{\frac{n+2}{n}},
\end{equation}
which implies that
\begin{equation}\label{ly-eigenvalue}
\left(\frac{n+2} {n}\right)^{n/2}\frac{ |\Omega|\omega(n)}{(2\pi)^n} \lambda_k^{n/2} \ge k,
\end{equation}
see also Laptev \cite{La97}, and \cite{FL20,FL24a,FLP26,GJL26,GLW11,KVW09,Me02} for further improvements on the lower order term and  generalizations.
{ The inequality \eqref{ly-sum} is equivalent to corresponding inequality for
Riesz means of order one (cf. \cite{FLW23}), which was recently improved  by  Frank and Larson \cite{FL24a} to
certain definite exponent $0<\gamma<1$ for all convex domains ($\gamma=0$ corresponds to P\'olya's conjecture). }
Note that the estimate \eqref{ly-eigenvalue} and its improvements are still a distance away from P\'olya's conjecture.

Our main purpose of this paper is to establish an approach to P\'olya's conjecture on more general domains (than tiling domains and products).
More precisely, given a Lipschitz domain $\Omega$, we can choose a larger domain $T$, so that by using the monotone property of Dirichlet eigenvalues, one has
$$\mathcal{N}^D_{\Omega}(\lambda)\le \mathcal{N}^D_{T}(\lambda)-\mathcal{N}^D_{T\setminus\overline{\Omega}}(\lambda).$$
The P\'olya conjecture will follow if one can prove, on one side $T$ satisfies better eigenvalue estimates than P\'olya conjectured,
and one the other side, $\mathcal{N}^D_{T\setminus\overline{\Omega}}(\lambda)$ has a reasonable good lower bound.

On one side, we shall prove that strip-tiling domains and certain product domains satisfy better eigenvalue estimates than P\'olya conjectured.
On the other side, we shall prove on certain domains composed by cubes of different scales the eigenvalue counting functions do have a good lower bound.
These enable us to provide in all dimensions, $n\ge 2$, a class of domains that satisfy P\'olya's conjecture.

For general Lipschitz domains, we did not have a good control of lower bound of $\mathcal{N}^D_{T\setminus\overline{\Omega}}(\lambda)$,
and our method only yields a new proof of quantitative Courant type estimate for the remainder.

If one can further obtain better estimate on the lower bound of $\mathcal{N}^D_{T\setminus\overline{\Omega}}(\lambda)$,
then one can combine the refined estimate on $\mathcal{N}^D_{T}(\lambda)$ (see Remark \ref{product-23} below, which behaves better as $\lambda\to \infty$),
to reduce P\'olya's conjecture to a computational problem.
In this paper, we can only reduce the $\epsilon$-loss version of  P\'olya's conjecture  to a computational problem.

\subsection{P\'olya's conjecture up to $\epsilon$-loss and quantitative remainder estimate}
\hskip\parindent Let us start from  a definition for P\'olya's conjecture up to an $\epsilon$-loss.
\begin{defn}\label{defn-epsilon-loss}
Let $\Omega\subset\rn$ be a bounded open domain, $n\ge 2$. For $\epsilon\in (0,1)$, we say that {\em P\'olya's conjecture up to an $\epsilon$-loss} or {\em $\epsilon$-loss version of P\'olya's conjecture}
holds for some or all Dirichlet eigenvalues $\lambda_k(\Omega)$, if it holds for some or all $k\in\cn$ that
\begin{align*}
k&\le  (1+\epsilon)\frac{|\Omega|\omega(n)}{(2\pi)^n}\lambda_k(\Omega)^{n/2}.
\end{align*}
\end{defn}
Noting that $2\le \left(\frac{n+2} {n}\right)^{n/2}\to e$ as $2\le n\to\infty$, the $\epsilon$-loss version of P\'olya's conjecture improves substantially
\eqref{ly-eigenvalue} (and its improvements on lower order terms, see \cite{Berezin72,GLW11,KVW09,Me02}).

For a bounded open domain $\Omega$, we define its width $\mathrm{width}(\Omega)$ as the minimum distance between two parallel hyperplanes that contain $\Omega$ inside.
Moreover, up to a rotation, throughout the paper we shall assume  the width $\mathrm{width}(\Omega)$ is attained in the $n$-th direction.
For a bounded Lipschitz domain $\Omega$, we let $C_{Lip}(\Omega)\ge 1$ be the constant, uniquely determined by $\Omega$, such that
for $\epsilon\in (0,\mathrm{width}(\Omega))$, such that
\begin{align}\label{constant-omega}
\left|\{x\in \rn:\, \dist(x,\partial\Omega)\le \epsilon\}\right| \le C_{Lip}(\Omega)\epsilon|\partial\Omega|.
\end{align}

We have the following P\'olya's conjecture up to $\epsilon$-loss.
\begin{thm}\label{epsilon-loss-lip}
Let $\Omega\subset\rr^n$ be a bounded Lipschitz domain, $n\ge 2$.
For any $0<\epsilon<1$,  it holds for all $\lambda_k(\Omega)\ge \Lambda(\epsilon,\Omega)$ that
\begin{align*}
k&\le  (1+\epsilon)\frac{|\Omega|\omega(n)}{(2\pi)^n}\lambda_k(\Omega)^{n/2},
\end{align*}
where $\Lambda(\epsilon,\Omega)\ge \mathrm{width}(\Omega)^{-2}$ is such that
\begin{align}\label{epsilon-constant}
\frac{2n \mathrm{diam}(\Omega)^{n-1}+|\partial\Omega|}{|\Omega|\sqrt{\Lambda(\epsilon,\Omega)}}C_{Lip}(\Omega) \left(\frac{\pi}{2}+5n^{2}\pi\log_2\frac{10\mathrm{width}(\Omega)\sqrt{\Lambda(\epsilon,\Omega)}}{\pi} \right)= \epsilon.
\end{align}
Moreover, if $\Omega$ is convex, then  $\Lambda(\epsilon,\Omega)$ can be taken smaller as
\begin{align}\label{epsilon-constant-convex}
\frac{2n \mathrm{diam}(\Omega)^{n-1}}{|\Omega|\sqrt{\Lambda(\epsilon,\Omega)}} \left(\frac{\pi}{2}+5n^{2}\pi\log_2\frac{10\mathrm{width}(\Omega)\sqrt{\Lambda(\epsilon,\Omega)}}{\pi} \right)= \epsilon.
\end{align}
\end{thm}
Obviously, for a given $\epsilon\in (0,1)$, and a concrete Lipschitz (resp. convex) domain $\Omega$,
$\Lambda(\epsilon,\Omega)$ is uniquely determined by the equation \eqref{epsilon-constant} (resp. \eqref{epsilon-constant-convex} for convex domains).
{From Theorem \ref{weyl-remainder-uniform} below and its proofs in Section 4, one believes that $\Omega$ needs only to be a bounded set of finite perimeter, but the constant $C_{Lip}(\Omega)$ (calculatable) might need to replaced by another less explicit geometric constant.}  Our aim here is to demonstrate a viability that to show the $\epsilon$-loss version of P\'olya's conjecture,
Theorem \ref{epsilon-loss-lip} reduces to a computational problem for eigenvalues $\lambda_k(\Omega)<\Lambda(\epsilon,\Omega)$.

Theorem \ref{epsilon-loss-lip}  follows directly from our following uniform estimate for the remainder of the Weyl's law.
Let us start from the following definition.
\begin{defn}[Minimal Admissible rectangle]\label{defn-adm-rectangle}
Let $\mathscr{R}=I_1\times I_2\times\cdots \times I_n$ be a rectangle in $\rr^n$, $n\ge 2$, where $I_i$ is connected open interval in $\rr$ for each $1\le i\le n$.
We say that $\mathscr{R}=I_1\times I_2\times\cdots \times I_n=\mathscr{R}_{n-1}\times I_{min}$ is the minimal admissible rectangle for $\Omega$, if $I_{min}=I_n=\mathrm{width}(\Omega)$,
and $\mathscr{R}$  contains $\Omega$, and for each $i=1,\cdots,n$,
for any $\tilde I_i$ with $|\tilde I_i|<|I_i|$, $I_1\times \cdots\tilde I_i\times\cdots \times I_n$ does not contain $\Omega$ inside.
\end{defn}


We have the  following uniform estimate on Lipschitz domains.
In what follows, for a bounded open domain $\Omega$, we denote by $r_{in}=r_{in}(\Omega)$ the radius of the largest ball contained in $\Omega$.

\begin{thm}\label{weyl-remainder-uniform}
Let $\Omega\subset\rr^n$ be a bounded Lipschitz domain, $n\ge 2$.  Let $\mathscr{R}=\mathscr{R}_{n-1}\times I_{min}$ be the {\em minimal Admissible rectangle} for $\Omega$. It holds for all $k\in\cn$ that
\begin{align*}
k&\le \frac{|\Omega|\omega(n)}{(2\pi)^n}\lambda_k(\Omega)^{n/2}+C(n,\Omega,\lambda_k(\Omega))\frac{\omega(n)}{(2\pi)^n}\lambda_k(\Omega)^{\frac{n-1}{2}},
\end{align*}
where
\begin{align}\label{lipschitz-uniform}
C(n,\Omega,\lambda_k(\Omega))= C_{Lip}(\Omega)|\partial(\mathscr{R}\setminus\overline{\Omega})| \left(\frac{\pi r_{in}(\mathscr{R}\setminus\overline{\Omega})}{|I_{min}|}+5n^{2}\pi\log_2\frac{10|I_{min}|\sqrt{\lambda_k(\Omega)}}{\pi} \right)-C_1(n-1)|\mathscr{R}_{n-1}|.
\end{align}
Moreover, it holds for each $k\in\cn$ that
\begin{align}\label{lower-uniform}
k&\ge
 \frac{\omega(n)|\Omega|}{(2\pi)^n} \lambda_k(\Omega)^{n/2}-  \frac{C_{Lip}(\Omega)\omega(n)|\partial\Omega| }{(2\pi)^n}\lambda_k(\Omega)^{\frac{n-1}{2}}\left(\frac{\pi r_{in}(\Omega)}{|I_{min}|}+5n^{2}\pi\log_2\frac{10|I_{min}| \sqrt{\lambda_k(\Omega)}}{\pi } \right).
\end{align}
Above,
\begin{align}\label{defn-c1n}
C_1(n)=\begin{cases}
\frac 23 & n=1,\\
\frac{3\pi}{16} & n=2,\\
\frac{\pi\mathscr{B}(\frac{n}{2}-1,2)}{6\mathscr{B}(\frac{n}{2}-1,5/2)} & n=3,4,\\
\frac{3\pi\mathscr{B}(\frac{n }{2}-1,2)}{16 \mathscr{B}(\frac{n }{2}-1,5/2)}\left(\frac{247}{256}\right)^{\frac{n-2}{2}}  &  n\ge 5.
\end{cases}
\end{align}
\end{thm}

Obviously, both quantities $ \frac{ r_{in}(\mathscr{R}\setminus\overline{\Omega})}{|I_{min}|}$ and $\frac{ r_{in}(\Omega)}{|I_{min}|}$ are not larger than $1/2$,
and therefore can be replaced by $1/2$. Moreover, note that if $|\partial(\mathscr{R}\setminus\overline{\Omega})|$ is small compared to $|\mathscr{R}_{n-1}|$,
we then have at least that P\'olya's conjecture holds for the first some eigenvalues.

{Let us emphasize that the quantitative Courant type estimate for the remainder has been already obtained in \cite{NS05} (see also \cite{FLW23}) via Courant's method of Dirichlet-Neumann bracketing, and also in \cite{FL24b} via refined Tauberian theory.
Our method, which depends on a modified min-max argument, is related to but different from the original ones from \cite{courant20,CH89,FLW23,NS05}. By using the {\em minimal Admissible rectangles},
we do not need to use Neumann eigenvalues. The latter is crucial in Courant's Dirichlet-Neumann bracketing method; see the proofs in
Section 4.
}


For convex sets, we have a clean and better estimate as following.
\begin{cor}\label{cor-uniform-weyl}
Let $\Omega\subset\rr^n$ be a bounded convex domain, $n\ge 2$.  Let $\mathscr{R}=\mathscr{R}_{n-1}\times I_{min}$ be the {\em minimal Admissible rectangle} for $\Omega$. It holds for all $k\in\cn$ that
\begin{align*}
k&\le \frac{|\Omega|\omega(n)}{(2\pi)^n}\lambda_k(\Omega)^{n/2}\\
&\quad+  \frac{\omega(n)}{(2\pi)^n}\lambda_k(\Omega)^{\frac{n-1}{2}} \left(|\partial\mathscr{R}\setminus\partial\Omega|\left(\frac{\pi r_{in}(\mathscr{R}\setminus{\overline{\Omega}})}{|I_{min}|}+5n^{2}\pi\log_2\frac{10|I_{min}|\sqrt{\lambda_k(\Omega)}}{\pi} \right)-C_1(n-1)|\mathscr{R}_{n-1}|\right),
\end{align*}
and
\begin{align*}
k&\ge
 \frac{\omega(n)|\Omega|}{(2\pi)^n} \lambda_k(\Omega)^{n/2}-  \frac{\omega(n)|\partial\Omega| }{(2\pi)^n}\lambda_k(\Omega)^{\frac{n-1}{2}}\left(\frac{\pi r_{in}(\Omega)}{|I_{min}|}+5n^{2}\pi\log_2\frac{10|I_{min}| \sqrt{\lambda_k(\Omega)}}{\pi } \right).
\end{align*}
\end{cor}

Let us remark that the logarithm term in \eqref{lipschitz-uniform} comes from the Whitney decomposition for $\mathscr{R}\setminus\overline{\Omega}$, see the proof of Theorem \ref{weyl-remainder-uniform} in Section 4 below.
For an explicitly given Lipschitz domain $\Omega$,
one may perform a better decomposition of $\mathscr{R}\setminus\overline{\Omega}$, which is not necessarily of Whitney type, and which may give better constant in the lower order terms if one can decompose the given domain into fewer and larger cubes or rectangles. Generally, to get better estimates than \eqref{lipschitz-uniform},
one has to find some generic and better decomposition than the Whitney decomposition we used in the proof.
If $\mathscr{R}\setminus\overline{\Omega}$ admits a better decomposition, such as  a union of up to infinite many cubes with finite surface, we then can prove full P\'olya's conjecture, see discussions in the next subsection.

\subsection{An approach to P\'olya's conjecture and concrete examples}
\hskip\parindent Using the strategy developed above, we can provide in all dimensions $n\ge 2$,
a class of domains which may have rather irregular shape or boundary but satisfy P\'olya's conjecture.
Using Seeley's asymptotic estimates \cite{Seeley78,Seeley80} (see also H\"ormander \cite{Hormander3}), we first provide
an approach to P\'olya's conjecture that might be  applied to general smooth domains.

\begin{defn}[Strip-tiling domains]
Let $\Omega$ be a bounded Lipschitz domain in $\rn$, $n\ge 2$.
Up to rotations and translations, we assume that $\Omega$ is contained in two hyperplanes  $\mathbb{R}^{n-1}\times \{0\}$ and $\mathbb{R}^{n-1}\times \{w(\Omega)\}$ for some $w(\Omega)\in (0,\infty)$.
We say that $\Omega$ is a strip-tiling domain, if
one can use isometries of $\Omega$, obtained by rotating $\Omega$ only in the first $n-1$ directions,
to cover $\mathbb{R}^{n-1}\times (0,w(\Omega))$ without overlapped interiors.
\end{defn}
Obviously, a rectangle $\mathscr{R}=\mathscr{R}_{n-1}\times I\subset\rn$, $n\ge 2$ is a strip-tiling domain.
The following Figure \ref{fig1} provide more  examples of general strip-tiling domains.
\begin{figure}[htbp]
\centering
\includegraphics[scale=0.6]{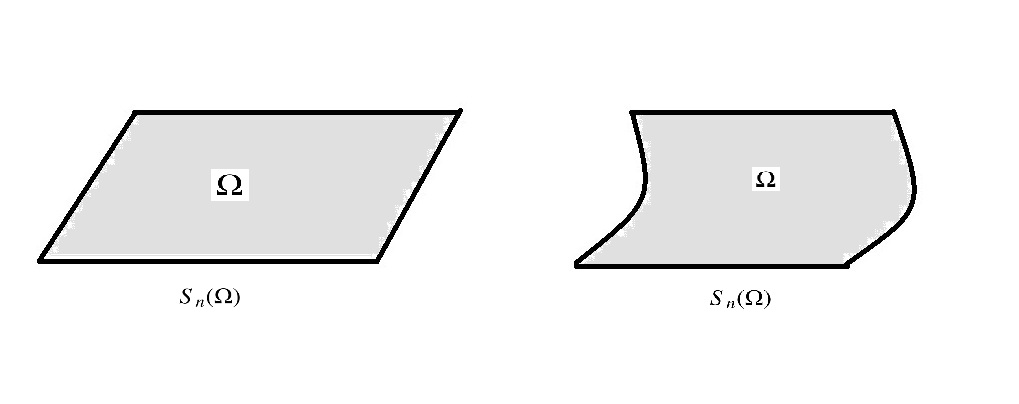}
\caption{Examples of strip-tiling domains in two dimension}
\label{fig1}
\end{figure}
\begin{rem}\label{1.8-strip}\rm
(i) Note that in the covering we fix the $n$-th  direction of $\Omega$, but allow rotating $\Omega$ in the first $n-1$ directions.
This requires the two side face of $\Omega$ in the $n$-th direction both tile $\rr^{n-1}$,
moreover, the two side face have the same surface measure.

(ii) Let $\Omega$ be a strip-tiling domain that can cover $\mathbb{R}^{n-1}\times (0,w(\Omega))$.
Let us denote by $S_n(\Omega)$ one of the two side faces of $\Omega$ in the $n$-th direction in what follows. Then the measure of $\Omega$
equals $|S_n(\Omega)|\times w(\Omega)$.
\end{rem}

\begin{thm}\label{qualitative-polya-rectangle}
Suppose that $\Omega$  is a strip-tiling domain  in $\rn$, $n\ge 2$. Let $S_n(\Omega)$ be a side face of $\Omega$ in the $n$-th direction.
Assume that $\Omega_0$ is a smooth domain contained in $\Omega$.
There exists $C_P=C(\Omega_0,n)$ such that if $|S_n(\Omega)|\ge C_P$ then the P\'olya conjecture holds on $\Omega\setminus\overline{\Omega}_0$.
\end{thm}

Our first observation is that eigenvalues on strip-tiling domains satisfying better estimates than P\'olya conjectured.
\begin{thm}\label{main}
Let $\Omega$ be a strip-tiling domain that can tile $\rr^{n-1}\times (0,w(\Omega))$, $n\ge 2$. Let $S_n(\Omega)$ be a side face of $\Omega$ in the $n$-th direction.
Then it holds that for all $k\ge 1$ that
\begin{align*}
k&\le \frac{|\Omega|\omega(n)}{(2\pi)^n}\lambda_k(\Omega)^{n/2}-C_1(n-1)|S_n(\Omega)| \frac{\omega(n)}{(2\pi)^n}\lambda_k(\Omega)^{\frac{n-1}{2}},
\end{align*}
where $C_1(n)$ is as in \eqref{defn-c1n}.
\end{thm}

For $M\ge 2$, by $\tilde\Omega$ $M$-tiles $\Omega$
we mean there are $M$ isometries of $\tilde\Omega$, $\{\tilde\Omega_{k}\}_{1\le k\le M}$, with non-overlapped interior, such that $\cup_{1\le k\le M}\tilde\Omega_k\subset \Omega$
and $|\Omega|=M|\tilde\Omega|$.

\begin{cor}
Suppose that $\Omega$ $K$-tiles a  strip-tiling domain $\tilde\Omega$ in $\rn$, $K\ge 2$ and $n\ge 2$.
 Let $S_n(\tilde\Omega)$ be a side face of $\tilde\Omega$ in the $n$-th direction.
Then it holds for all $k\ge 1$ that
\begin{align*}
k&\le \frac{|\Omega|\omega(n)}{(2\pi)^n}\lambda_k(\Omega)^{n/2}-\frac{C_1(n-1)}{K}|S_n(\tilde\Omega)| \frac{\omega(n)}{(2\pi)^n}\lambda_k(\Omega)^{\frac{n-1}{2}}.
\end{align*}
\end{cor}

The above result applies especially to any triangles in the plane.
\begin{cor}
Let $\Omega$ be a triangle in $\rr^2$ with vertices at $a,b,c$.
Then it holds for all $k\ge 1$ that
\begin{align*}
k&\le \frac{|\Omega|}{4\pi}\lambda_k(\Omega)-\frac{\max\{|ab|,|bc|,|ca|\}}{12\pi} \lambda_k(\Omega)^{\frac{1}{2}}.
\end{align*}
\end{cor}
\begin{rem}\rm
Note that previously, it was only known that on an arbitrary triangle $\Omega$,
\begin{align*}
k&\le \frac{|\Omega|}{4\pi}\lambda_k(\Omega), \ \forall\ k\ge 1,
\end{align*}
which is due to P\'olya \cite{Polya61}.
In special cases of equilateral triangle, see Pinsky \cite{Pinsky80}.
\end{rem}

In what follows, by a cube $Q$ we always mean it is an open cube and  denote by $\ell(Q)$ its side length.
Let us first introduce a definition. Let us emphasize  that, since the P\'olya conjecture is scaling invariant,
it is not restrictive to require in the following definition that the side length of the largest cube $Q$ is  one.

\begin{defn}[Admissible class]\label{defn-cube}
Let $n\ge 2$. Let $\{Q_k\}_{k\in \cn}$ be a sequence of cubes in $\rn$, where
we allow the set $\{Q_k\}_{k\in \cn}$ to be of finitely many cubes $\{Q_k\}_{1\le k\le k_0}$, by simply setting $Q_k$ as empty set for $k>k_0$.
Assume the largest cube in $\{Q_k\}_{k\in \cn}$ is of side length one.
Set $\mathscr{S}_Q$ as the surface area of these cubes, i.e.,
$$\mathscr{S}_Q:=\sum_{k\in \cn} \ell(Q_k)^{n-1}.$$
We say that $\{Q_{k}\}_{k\in \cn}$ is of {\em Admissible class} if $\mathscr{S}_Q<\infty$, and for any $k,j\in \cn$, $Q_k\cap Q_j=\emptyset$ if $k\neq j$.
\end{defn}
We have a class of domains that satisfy P\'olya's conjecture in all dimensions.
\begin{thm}\label{polya-rectangle-minus}
Let $\Omega$ be a strip-tiling domain in $\rn$, $n\ge 2$. Let $S_n(\Omega)$ be  a side face of $\Omega$ in the $n$-th direction.
Let  $\{Q_{k}\}_{k\in \cn}$ be of {\em Admissible class} in $\rn$, that is contained in $\Omega$.
If $|S_n(\Omega)|\ge C_2(n-1) \mathscr{S}_Q$, then the P\'olya conjecture holds on $\Omega\setminus\overline{\cup_{k\in \cn}Q_{k}}$,
where
\begin{align}\label{defn-c2n}
C_2(n)=\begin{cases}
2\sqrt 2\pi,& n=1,\\
9\sqrt 3, & n=2,\\
\frac{6(n+1)^{3/2}\mathscr{B}(\frac{n}{2}-1,5/2)}{\mathscr{B}(\frac{n}{2}-1,2)} & n=3,4,\\
\frac{16(n+1)^{3/2} \mathscr{B}(\frac{n }{2}-1,5/2)}{3\mathscr{B}(\frac{n }{2}-1,2)}\left(\frac{256}{247}\right)^{\frac{n-2}{2}}  &  n\ge 5.\end{cases}
\end{align}
\end{thm}

\begin{cor}\label{2d-polya-cor}
Let $\Omega$ be a strip-tiling domain in $\rn$, $n\ge 2$. Let $S_n(\Omega)$ be  a side face of $\Omega$ in the $n$-th direction. Let  $M\ge 2$
and $\Omega_0$ $M$-tile $\Omega$.
Let  $\{Q_{k}\}_{k\in \cn}$ be of {\em Admissible class} in $\rn$, that is contained in $\Omega_0$.
If $|S_n(\Omega)|\ge C_2(n-1) M\mathscr{S}_Q$, then P\'olya's conjecture holds on $\Omega_0\setminus\overline{\cup_{k\in \cn}Q_{k}}$.
\end{cor}
Note that $M$ isometries of the domain $\Omega_0\setminus \overline{\cup_{k\in \cn}Q_{k}}$
corresponds to the case that removing $M$ copies of  $\{Q_{k}\}_{k\in \cn}$  from $\Omega$, the conclusion is obvious.

Some further remarks are in order. In what follows for $C>0$ we use $\lfloor C \rfloor$ to denote the integer part of it.
\begin{rem}\label{rem-examples}\rm
(i)  Let us take $\ell(k)=2^{\lfloor k/2\rfloor}$ for $k\ge 2$, $\ell(1)$=1,
and $\{Q_{1,1},\, \{Q_{k,j}\}_{k\ge 2, 1\le j\le \ell(k)}\}$ be a sequence of cubes, where $\ell(Q_{k,j})=2^{1-k}$.
Then $\mathscr{S}_Q\le 2(2+\sqrt 2)$.  So for a strip-tiling domain $\Omega$ with $|S_n(\Omega)|> 8(1+\sqrt 2)\pi$,
one can arbitrarily remove any domain, that is composed by all or some of these cubes $\{Q_{1,1},\, \{Q_{k,j}\}_{k\ge 2, 1\le j\le \ell(k)}\}$ with non-overlapped interior,
from $\Omega$ such that the remaining set satisfies P\'olya's conjecture. See Figure \ref{fig2} for some examples.

\begin{figure}[htbp]
    \centering

    \subfloat[A strip-tiling domain with a diamond and a room-corridor  removed]
    {
        \begin{minipage}[t]{0.5\textwidth}
            \centering
            \includegraphics[width=0.92\textwidth]{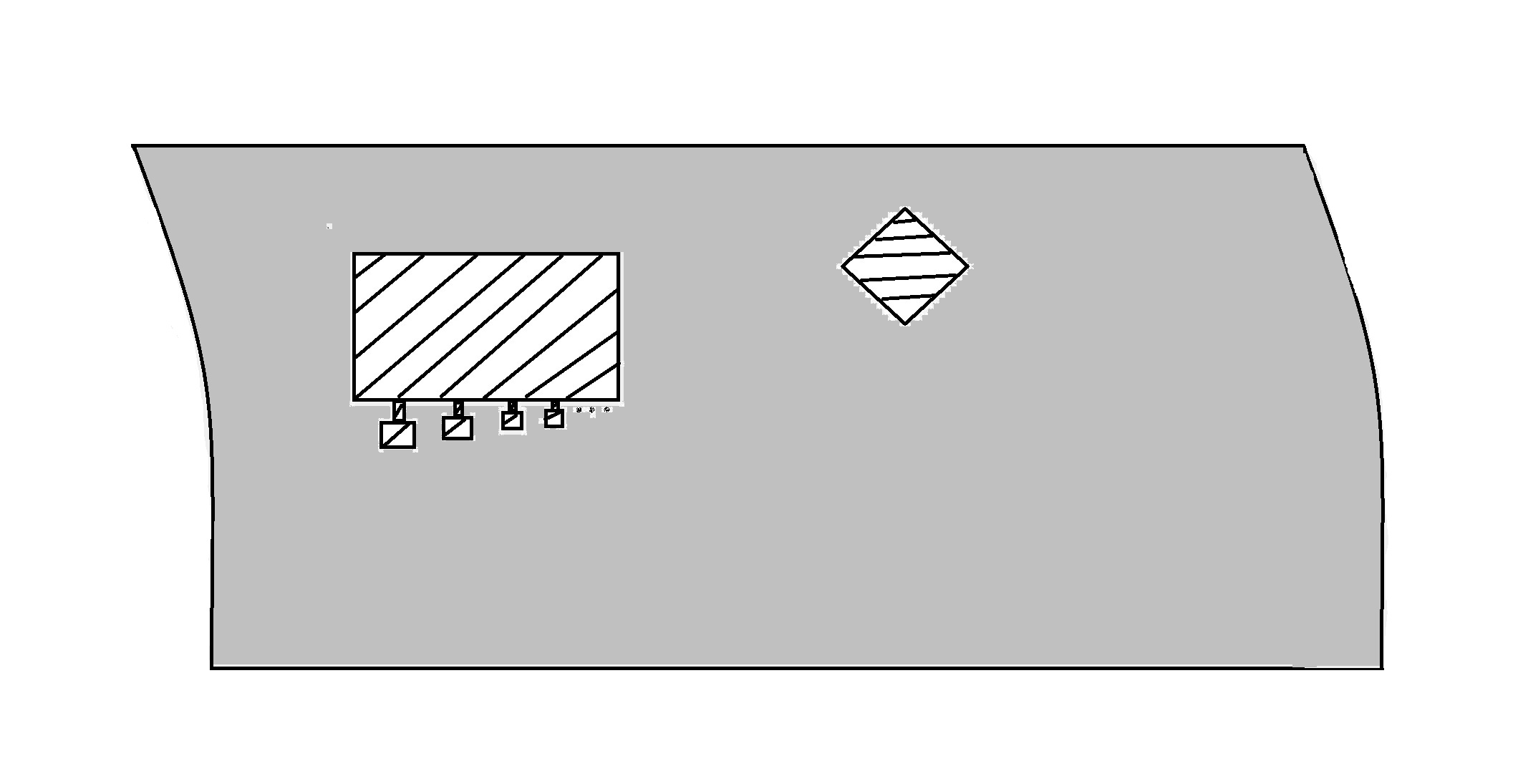}
        \end{minipage}%
    }
    \subfloat[Trapezoid with infinite cubes removed]
    {
        \begin{minipage}[t]{0.5\textwidth}
            \centering
            \includegraphics[width=0.9\textwidth]{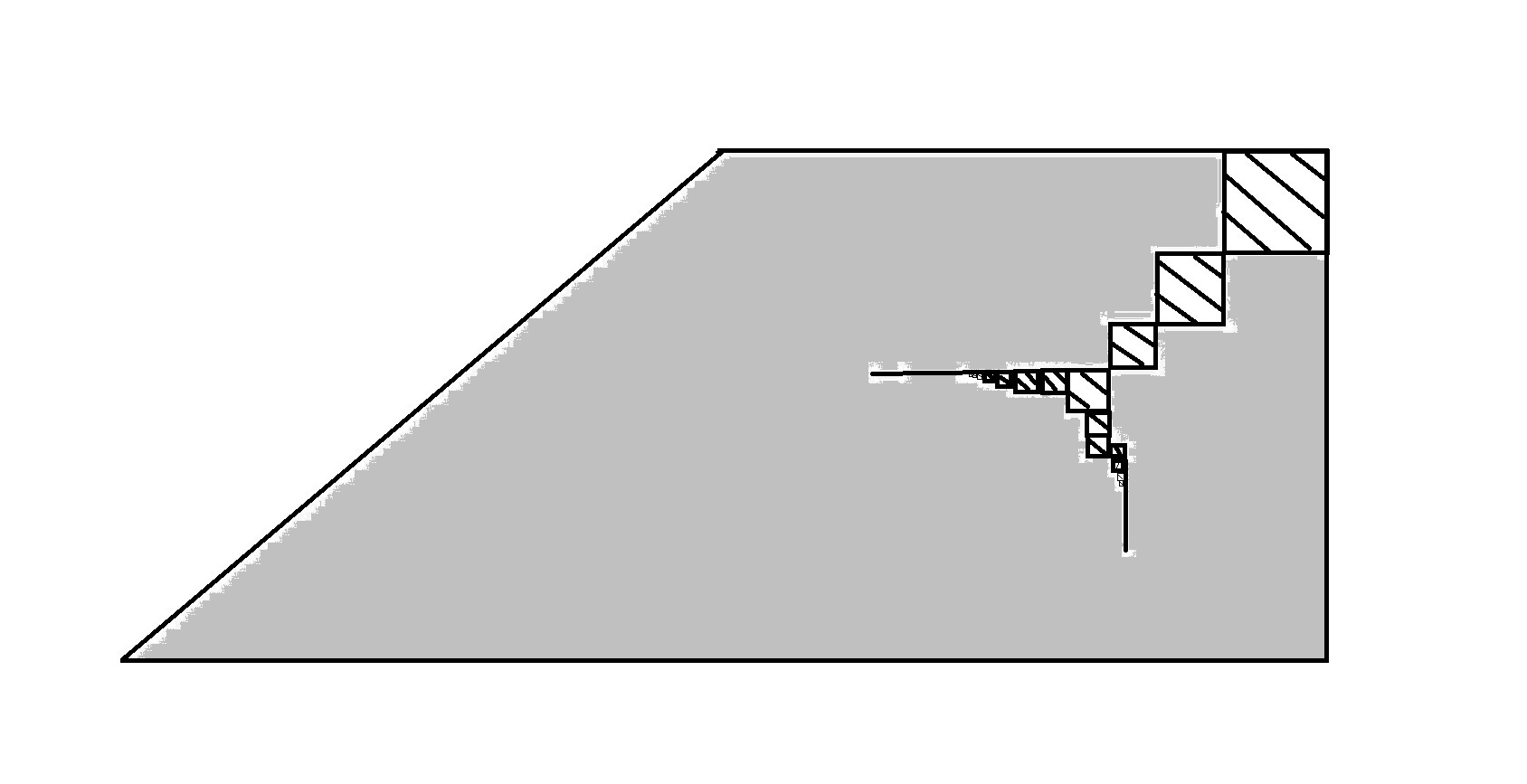}
        \end{minipage}
    }%

    \caption{Examples satisfying P\'olya's conjecture}
    \label{fig2}
\end{figure}

(ii) Previously, in $\rr^2$, P\'olya's conjecture was only known to be true on two types of domains,  tiling domains (\cite{Polya61}), balls, sectors and annuli (\cite{FLPS23,FLPS25}).
In higher dimensions $n\ge 3$, other than tiling domains, balls and sectors that tile balls, there are a class of product domains satisfying   P\'olya's conjecture (\cite{La97}),
which are products of open sets in $\rr^{n-k}$ and domains satisfying P\'olya's conjecture in $\rr^{k}$ (tiling domains, balls or sectors that tile balls), $k\ge 2$. See also \cite{HW24} for a specific product in $\rr^3$.

Our result (Theorem \ref{polya-rectangle-minus}) provides in all dimensions $n\ge 2$, a class of domains that may have rather irregular shapes or boundaries but satisfy P\'olya's conjecture.
\end{rem}

\subsection{Our method and plan of the paper} \hskip\parindent
The main strategy of the paper is as following.

We first use the explicit formula of Riesz means of Dirichlet eigenvalues
in $\rr$ (cf. Lemma \ref{lem-riesz-dirichlet-n1}), following the inspiring observation by Laptev \cite{La97},
 to give a refined estimates on product domains $\Omega_1\times\Omega_2$ (see Theorem \ref{pertubation-product-1d} below).
 Using these refined estimates on product domains together with P\'olya's method \cite{Polya61},
 we give the refined estimates on strip-tiling domains.

We then establish a quantitative lower bound for the counting function (upper bound for the Dirichlet eigenvalue) on cubes (see Lemma \ref{opposite-cube} and Lemma \ref{opposite-cube-hign} below), which together with our refinements on strip-tiling domains and
monotonicity of Dirichlet eigenvalues  to conclude the proof Theorem \ref{polya-rectangle-minus}. Moreover, in higher dimensions, $n\ge 3$, certain product domains can replace the role of strip-tiling domains, see Theorem \ref{3hd-polya} below.

For Theorem \ref{weyl-remainder-uniform} we shall apply the Whitney decomposition of a Lipschitz domain $\Omega$, and using the lower bound for the counting function
on cubes ({see Lemma \ref{opposite-cube} and Lemma \ref{opposite-cube-hign} below}) along with an optimization to deduce a lower bound for the counting function  $\mathcal{N}_\Omega^D(\lambda)$ on $\Omega$.
Finally, applying our refinements on product domains ({see Theorem \ref{pertubation-product-1d} below}) and
monotonicity of Dirichlet eigenvalues, together with the lower bound of $\mathcal{N}_{\mathscr{R}\setminus\overline{\Omega}}^D(\lambda)$ to conclude the proof. {Theorem \ref{epsilon-loss-lip}
then follows almost immediately from Theorem \ref{weyl-remainder-uniform}. }

%
%
%

The paper is organized as follows.

In Section 2, we study the refined estimates on product domains, and prove the main results
 Theorem \ref{pertubation-product-1d}  and Corollary \ref{first-k-dirichlet-1d}.

In Section 3, we provide the proofs for Theorem \ref{qualitative-polya-rectangle} and  the results on concrete examples of P\'olya's conjecture, Theorem \ref{polya-rectangle-minus}.
Since in higher dimensions, $n\ge 3$, we can replace the rectangle $\mathscr{R}_{n-1}$ by general set satisfying P\'olya's conjecture,
we shall split the proof into the case of two dimension, and three and higher dimensions where we prove a more general result, Theorem \ref{3hd-polya}.
We shall also  show some results regarding cubes of Admissible class from some more general class of domains,
see Theorem \ref{removing-domain-nonproduct} below.

We shall employ {Lemma \ref{opposite-cube} and Lemma \ref{opposite-cube-hign} from Section 3 and Theorem \ref{pertubation-product-1d} from Section 2} to prove the uniform estimates on Lipschitz sets, Theorem \ref{weyl-remainder-uniform} and its corollary, Theorem \ref{epsilon-loss-lip}.

For $C>0$ we use $\lfloor C \rfloor$ to denote the integer part of it. For $a\in \rr$ we denote by $a_+$ the quantity
 $\max\{0,a\}$. For a cube $Q$ we denote by $\ell(Q)$ its side length.
For a bounded open domain $\Omega$, we denote by $r_{in}=r_{in}(\Omega)$ the radius of the largest ball contained in $\Omega$.
Throughout the paper, most of constants are given explicitly,
if not we shall point it out.

\section{Refined eigenvalue estimates}
\hskip\parindent In this section, we prove refined eigenvalue estimates on product domains and strip-tiling domains.

\subsection{Product domains}\hskip\parindent
We first prove Theorem \ref{pertubation-product-1d}  and its corollary,  Corollary \ref{first-k-dirichlet-1d}.
Recall that, if an open domain $\Omega$ is in the real line, then its eigenvalues have explicit lower bounds, which satisfy for $k\ge 1$
$$\lambda_k(\Omega)\ge \frac{\pi^2}{|\Omega|^2}k^2.$$
This allows us to show the following estimate for the Riesz means.
\begin{lem}\label{lem-riesz-dirichlet-n1}
 Let $\Omega\subset\rr$ be bounded. It holds that
\begin{align}\label{riesz-dirichlet-n1}
\sum_{k:\,\lambda_k(\Omega)<\lambda}(\lambda-\lambda_k)&\le
\frac{2|\Omega|}{3\pi} \lambda^{3/2} \min\left\{ \left(1 - \frac{3\pi}{16|\Omega|\sqrt\lambda} \right)_+,\,\left(1 - \left(\frac{\lfloor\frac{|\Omega|\sqrt{\lambda}}{\pi}\rfloor}{\frac{|\Omega|\sqrt{\lambda}}{\pi}}\right)^2\frac{3\pi}{4|\Omega|\sqrt\lambda} \right)\right\},
\end{align}
and for  $\lambda>\frac{\pi^2}{|\Omega|^2}$,
\begin{align}\label{riesz-dirichlet-n2}
\sum_{k:\,\lambda_k(\Omega)<\lambda}(\lambda-\lambda_k)^{1/2}&\le \frac{|\Omega|}{4}\lambda- \frac { \sqrt{\lambda}}{2 }+
  \frac{\sqrt 6\pi}{9|\Omega|} \sqrt{\lfloor\frac{|\Omega|}{\pi}\sqrt{\lambda}\rfloor+\frac 12}\le
 \frac {|\Omega|}{4} \lambda\left(1-\frac{2}{3} \frac{1}{|\Omega|\sqrt\lambda}\right).
 \end{align}
 Moreover, if $\Omega$ is a connected open interval, it holds for $\lambda>\lambda_1(\Omega)=\frac{\pi^2}{|\Omega|^2}$ that
 \begin{align}\label{riesz-dirichlet-n2-lower}
\sum_{k:\,\lambda_k(\Omega)<\lambda}(\lambda-\lambda_k)^{1/2}&\ge
 \frac{|\Omega|}{4}\lambda \left(1-\frac{2\pi}{|\Omega|\sqrt\lambda}+\frac{\pi^2}{|\Omega|^2\lambda}\right).
 \end{align}
\end{lem}
\begin{proof} Let us first prove \eqref{riesz-dirichlet-n1}.
If $\lambda\le \frac{\pi^2}{|\Omega|^2}$, then $\sum_{k:\,\lambda_k(\Omega)<\lambda}(\lambda-\lambda_k)=0$. Suppose $\lambda>\frac{\pi^2}{|\Omega|^2}$ and set $k_\lambda=\lfloor\frac{|\Omega|}{\pi}\sqrt{\lambda}\rfloor$. Then we have
\begin{align*}
\sum_{k:\,\lambda_k(\Omega)<\lambda}(\lambda-\lambda_k)&=\sum_{1\le k\le k_\lambda}(\lambda-\lambda_k)\le \sum_{1\le k\le k_\lambda}(\lambda-\frac{\pi^2}{|\Omega|^2}k^2)\\
&= \lambda k_\lambda - \frac{\pi^2}{|\Omega|^2}\frac 16 k_\lambda \left(k_\lambda+1\right) \left(2k_\lambda+1\right) \\
&= \lambda k_\lambda- \frac{\pi^2}{3|\Omega|^2} k_\lambda^3 - \frac{\pi^2}{2|\Omega|^2}k_\lambda^2
-\frac{\pi^2}{6|\Omega|^2}k_\lambda.
\end{align*}
Since
\begin{align*}
\lambda k_\lambda- \frac{\pi^2}{3|\Omega|^2} k_\lambda^3\le  \frac{2|\Omega|}{3\pi}\lambda^{3/2},
\end{align*}
we see that
\begin{align*}
\sum_{k:\,\lambda_k(\Omega)<\lambda}(\lambda-\lambda_k)&<  \frac{2|\Omega|}{3\pi}\lambda^{3/2}- \frac{\pi^2}{2|\Omega|^2}\left(\frac{k_\lambda }{\frac{|\Omega|}{\pi}\sqrt{\lambda}}\right)^2  \frac{|\Omega|^2}{\pi^2}\lambda
= \frac{2|\Omega|}{3\pi}\lambda^{3/2}- \left(\frac{k_\lambda}{\frac{|\Omega|}{\pi}\sqrt{\lambda}}\right)^2 \frac{\lambda}{2}.
\end{align*}
For $\lambda>\frac{\pi^2}{|\Omega|^2}$, $\frac{k_\lambda}{\frac{|\Omega|}{\pi}\sqrt{\lambda}}\ge 1/2$, so we also have
\begin{align*}
\sum_{k:\,\lambda_k(\Omega)<\lambda}(\lambda-\lambda_k)&< \frac{2|\Omega|}{3\pi}\lambda^{3/2}-  \frac{\lambda}{8}.
\end{align*}

Let us prove \eqref{riesz-dirichlet-n2}. Again we only need to consider $\lambda>\frac{\pi^2}{|\Omega|^2}$. Set $k_\lambda=\lfloor\frac{|\Omega|}{\pi}\sqrt{\lambda}\rfloor$. Note that
\begin{align*}
\sum_{k:\,\lambda_k(\Omega)<\lambda}(\lambda-\lambda_k)^{1/2}&\le \sum_{1\le k\le k_\lambda}(\lambda-\frac{\pi^2}{|\Omega|^2}k^2)^{1/2}=\frac{\pi}{|\Omega|}\sum_{1\le k\le k_\lambda}(\frac{|\Omega|^2}{\pi^2}\lambda-k^2)^{1/2}.
\end{align*}
Let us consider the ball $B(0,\frac{|\Omega|}{\pi}\sqrt{\lambda})$ in the first quadrant in $\rr^2$.
For each $1\le k\le k_\lambda=\lfloor\frac{|\Omega|}{\pi}\sqrt{\lambda}\rfloor$, $(\frac{|\Omega|^2}{\pi^2}\lambda-k^2)^{1/2}$
equals the area of the rectangle $[k-1,k]\times [0, (\frac{|\Omega|^2}{\pi^2}\lambda-k^2)^{1/2}]$, which is contained in $B(0,\frac{|\Omega|}{\pi}\sqrt{\lambda})$ in the first quadrant in $\rr^2$.
Moreover, for each $k$, there is at least a triangle with vertices at $(k-1, (\frac{|\Omega|^2}{\pi^2}\lambda-(k-1)^2)^{1/2})$,
$(k-1, (\frac{|\Omega|^2}{\pi^2}\lambda-k^2)^{1/2})$ and $(k, (\frac{|\Omega|^2}{\pi^2}\lambda-k^2)^{1/2})$
contained in the ball. Moreover, if $k_\lambda\neq\frac{|\Omega|}{\pi}\sqrt{\lambda} $, there is another triangle insider the 1/4 ball
with vertices at  $(k_\lambda, (\frac{|\Omega|^2}{\pi^2}\lambda-k_\lambda^2)^{1/2})$,
$(k_\lambda, (\frac{|\Omega|^2}{\pi^2}\lambda-k_\lambda^2)^{1/2})$ and $(\frac{|\Omega|}{\pi}\sqrt{\lambda},0)$. See Figure \ref{fig3}\footnote[2]{Picture drawn by Xinyi Jiang.}.
\begin{figure}[htbp]
\centering
\includegraphics[scale=0.2]{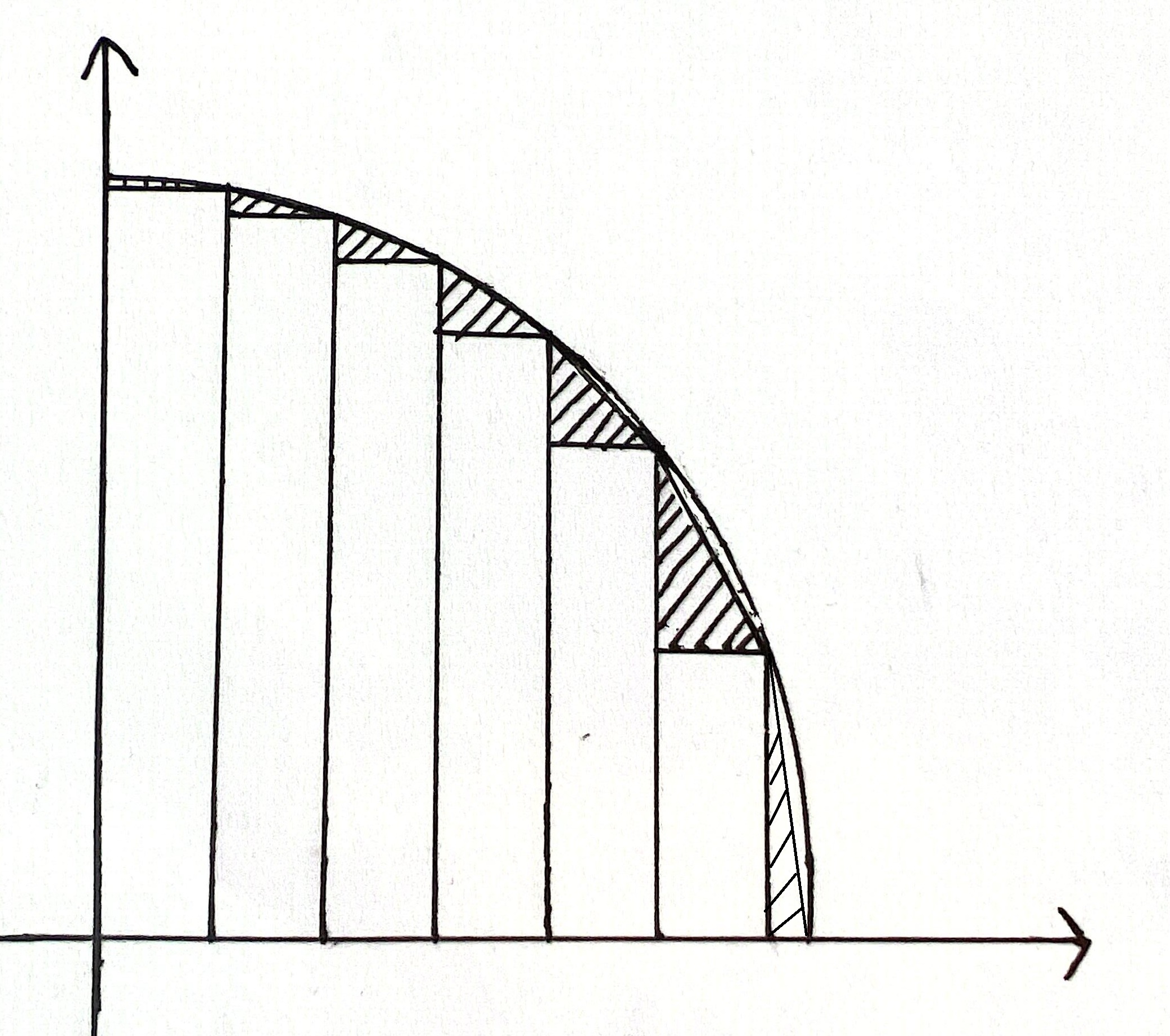}
\caption{The 1/2 sum is controlled by the 1/4 ball}
\label{fig3}
\end{figure}
We therefore see that
\begin{align*}
\frac {\pi}{4}\frac{|\Omega|^2}{\pi^2}\lambda&> \sum_{1\le k\le k_\lambda}(\frac{|\Omega|^2}{\pi^2}\lambda-k^2)^{1/2}+\frac 12\sum_{1\le k\le k_\lambda} 1\times \left((\frac{|\Omega|^2}{\pi^2}\lambda-(k-1)^2)^{1/2}-(\frac{|\Omega|^2}{\pi^2}\lambda-k^2)^{1/2}\right)\\
&\quad+ \frac 12 \left(\frac{|\Omega|}{\pi}\sqrt{\lambda}-k_\lambda\right)\times (\frac{|\Omega|^2}{\pi^2}\lambda-k_\lambda^2)^{1/2}\\
&=\sum_{1\le k\le k_\lambda}(\frac{|\Omega|^2}{\pi^2}\lambda-k^2)^{1/2}+\frac 12 \left((\frac{|\Omega|^2}{\pi^2}\lambda)^{1/2}-(\frac{|\Omega|^2}{\pi^2}\lambda-k_\lambda^2)^{1/2}+\left(\frac{|\Omega|}{\pi}\sqrt{\lambda}-k_\lambda\right)\times (\frac{|\Omega|^2}{\pi^2}\lambda-k_\lambda^2)^{1/2}\right).
\end{align*}
For $x\in [0,1)$, let us consider the function
$$g(x):=k_\lambda+x- \left((k_\lambda+x)^2-k_\lambda^2\right)^{1/2}(1-x).$$
A calculation shows that for $k_\lambda\in\cn$, it holds
\begin{align*}
g(x)&= k_\lambda+x -\left(2k_\lambda x+x^2\right)^{1/2}(1-x)\\
&\ge k_\lambda+x-\sqrt{2k_\lambda+1}\sqrt{x}(1-x)\\
&\ge k_\lambda+x-\frac{2\sqrt 3}{9} \sqrt{2k_\lambda+1}\\
&\ge k_\lambda+x-\frac{2}{3}k_\lambda\\
&\ge \frac{k_\lambda+x}{3}.
\end{align*}
This implies that
\begin{align*}
\frac {\pi}{4}\frac{|\Omega|^2}{\pi^2}\lambda&> \sum_{1\le k\le k_\lambda}(\frac{|\Omega|^2}{\pi^2}\lambda-k^2)^{1/2}+
\frac{1}{2}\frac{|\Omega|}{\pi}\sqrt{\lambda}-\frac{\sqrt 6}{9} \sqrt{\lfloor\frac{|\Omega|}{\pi}\sqrt{\lambda}\rfloor+\frac 12}\\
&\ge  \sum_{1\le k\le k_\lambda}(\frac{|\Omega|^2}{\pi^2}\lambda-k^2)^{1/2} +\frac{1}{6}\frac{|\Omega|}{\pi}\sqrt{\lambda},
\end{align*}
i.e.,
\begin{align*}
 \frac{\pi}{|\Omega|}\sum_{1\le k\le k_\lambda}(\frac{|\Omega|^2}{\pi^2}\lambda-k^2)^{1/2}< \frac{|\Omega|}{4}\lambda- \frac { \sqrt{\lambda}}{2 }+
  \frac{\sqrt 6\pi}{9|\Omega|} \sqrt{\lfloor\frac{|\Omega|}{\pi}\sqrt{\lambda}\rfloor+\frac 12}\le
 \frac {|\Omega|}{4} \lambda\left(1-\frac{2}{3} \frac{1}{|\Omega|\sqrt\lambda}\right).
\end{align*}

Let us prove \eqref{riesz-dirichlet-n2-lower}. Since $\Omega$ is connected, we have
$$\lambda_k(\Omega)=\frac{k^2\pi^2}{|\Omega|^2},\,\forall \ k\in\cn. $$
Set again $k_\lambda=\lfloor\frac{|\Omega|}{\pi}\sqrt{\lambda}\rfloor$. For $\lambda>\frac{\pi^2}{|\Omega|^2}$, we have
\begin{align*}
\sum_{k:\,\lambda_k(\Omega)<\lambda}(\lambda-\lambda_k)^{1/2}&= \sum_{1\le k\le k_\lambda}(\lambda-\frac{\pi^2}{|\Omega|^2}k^2)^{1/2}=\frac{\pi}{|\Omega|}\sum_{1\le k\le k_\lambda}(\frac{|\Omega|^2}{\pi^2}\lambda-k^2)^{1/2}.
\end{align*}
Note that the union of the rectangles $[k-1,k]\times [0, (\frac{|\Omega|^2}{\pi^2}\lambda-k^2)^{1/2}]$, $1\le k\le k_\lambda$, contains the ball
$B(0,\sqrt{\frac{|\Omega|^2}{\pi^2}\lambda-2\sqrt{\frac{|\Omega|^2}{\pi^2}\lambda}+1}) $ intersecting the first quadrant, since
\begin{align*}
\sqrt{(k-1)^2+ (\frac{|\Omega|^2}{\pi^2}\lambda-k^2)}&=\sqrt{\frac{|\Omega|^2}{\pi^2}\lambda-2k+1}\ge \sqrt{\frac{|\Omega|^2}{\pi^2}\lambda-2k_\lambda+1}\\
&\ge  \sqrt{\frac{|\Omega|^2}{\pi^2}\lambda-2\sqrt{\frac{|\Omega|^2}{\pi^2}\lambda}+1}.
\end{align*}
We therefore deduce that
\begin{align*}
\sum_{k:\,\lambda_k(\Omega)<\lambda}(\lambda-\lambda_k)^{1/2}&\ge \frac{\pi}{4}\frac{\pi}{|\Omega|} \left(\frac{|\Omega|^2}{\pi^2}\lambda-2\sqrt{\frac{|\Omega|^2}{\pi^2}\lambda}+1\right)\\
&= \frac{|\Omega|}{4}\lambda \left(1-\frac{2\pi}{|\Omega|\sqrt\lambda}+\frac{\pi^2}{|\Omega|^2\lambda}\right).
\end{align*}
The proof is complete.
\end{proof}

\begin{lem}\label{lem-1d}
Let $\Omega\subset\rr$ be an open bounded set. Then for $p>1$ and $\lambda>\lambda_1=\lambda_1(\Omega)$, it holds that
\begin{align*}
\sum_{k:\,\lambda_k<\lambda } (\lambda-\lambda_k)^{p}
&\le  \frac{2|\Omega|}{3\pi\mathscr{B}(p-1,2)}  \lambda^{p+\frac{1}{2}} \mathscr{B}(p-1,5/2)- C_3(p)\lambda^p,
\end{align*}
where
\begin{align}\label{defn-c3n}
C_3(p)=
\begin{cases}
\frac 18,& p=1,\\
\frac 19, & 1<p\le 2,\\
\frac{1}{8 }\left(\frac{247}{256}\right)^{p-1}, & p>2.
\end{cases}
\end{align}
\end{lem}
\begin{proof}
By the semigroup type property of the Riesz means and \eqref{riesz-dirichlet-n1}, we have
\begin{align}
\sum_{k:\,\lambda_k<\lambda } (\lambda-\lambda_k)^{p}
&\le \mathscr{B}(p-1,2)^{-1} \sum_{k}  \int_0^\infty v^{p-2}(\lambda-\lambda_k-v)_+\,dv \nonumber\\
&\le  \frac{2|\Omega|}{3\pi\mathscr{B}(p-1,2)} \int_0^\infty v^{p-2} (\lambda-v)_+^{3/2}\left(1 - \frac{3\pi}{16|\Omega|\sqrt{(\lambda-v)_+}} \right)_+\,dv\nonumber\\
&\le \frac{2|\Omega|}{3\pi\mathscr{B}(p-1,2)}  \lambda^{p+\frac{1}{2}} \mathscr{B}(p-1,5/2)- \frac{1}{8\mathscr{B}(p-1,2)}    \int_0^{\lambda-\frac{9\pi^2}{256|\Omega|^2}} v^{p-2} (\lambda-v)_+\,dv.
\end{align}
Noticing  that
$\lambda_1(\Omega)\ge \frac{\pi^2}{|\Omega|^2},$
so for $\lambda>\lambda_1(\Omega)$,
$$\lambda-\frac{9\pi^2}{256|\Omega|^2}\ge \frac{247}{256}\lambda.$$
For $1<p\le 2$,
\begin{align}
\frac{1}{8\mathscr{B}(p-1,2)}  \int_0^{\lambda-\frac{9\pi^2}{256|\Omega|^2}} v^{p-2} (\lambda-v)_+\,dv\ge \frac{\lambda^p}{8\mathscr{B}(p-1,2)}  \frac{249}{256} \mathscr{B}(p-1,2)\ge  \frac{\lambda^p}{9}.
\end{align}
For $p>2$, we have
\begin{align*}
  \int_{\lambda-\frac{9\pi^2}{256|\Omega|^2}}^\lambda v^{p-2} (\lambda-v)_+\,dv &= \frac{\lambda^p}{p-1}- \frac{\lambda^p}{p}-\left(\frac{\lambda (\lambda-\frac{9\pi^2}{256|\Omega|^2})^{p-1}}{p-1}-\frac{(\lambda-\frac{9\pi^2}{256|\Omega|^2})^p}{p} \right)\\
&=\frac{\lambda^p}{p(p-1)}- \left(\lambda-\frac{9\pi^2}{256|\Omega|^2}\right)^{p-1} \left(\frac{\lambda}{p-1}-\frac{(\lambda-\frac{9\pi^2}{256|\Omega|^2})}{p} \right)\\
&\le \frac{\lambda^p}{p(p-1)} -\left(\lambda-\frac{9\pi^2}{256|\Omega|^2}\right)^{p-1} \frac{\lambda }{p(p-1)}
\end{align*}
and therefore
\begin{align*}
&\frac{1}{8\mathscr{B}(p-1,2)}    \int_0^{\lambda-\frac{9\pi^2}{256|\Omega|^2}} v^{p-2} (\lambda-v)_+\,dv\\
&= \frac{1}{8\mathscr{B}(p-1,2)}   \left( \int_0^{\lambda} v^{p-2} (\lambda-v)_+\,dv- \int_{\lambda-\frac{9\pi^2}{256|\Omega|^2}}^\lambda v^{p-2} (\lambda-v)_+\,dv\right)\\
&\ge \frac{1}{8\mathscr{B}(p-1,2)} \left(\lambda-\frac{9\pi^2}{256|\Omega|^2}\right)^{p-1} \frac{\lambda }{p(p-1)}\\
&\ge \frac{\lambda^p}{8 }\left(\frac{247}{256}\right)^{p-1}.
\end{align*}
Combining the above estimates yields the desired conclusion.
\end{proof}

In view of Laptev's beautiful theory on product domains \cite{La97} (which has been further employed in Laptev-Weidl \cite{LW00} for Schr\"odinger operator), the Aizenman-Lieb principle \cite{AiLi78},
we have the following
refined  eigenvalue estimates on product domains if one domain satisfies P\'olya's conjecture.
 We note that Larson \cite{Larson17} earlier had obtained similar improvements for higher dimensions ($n_1\ge 3$ below) for more general products but with less explicit constants.
Note that any bounded domain $\Omega\subset\rr$ satisfies P\'olya's conjecture, i.e.,
$\lambda_k(\Omega)\ge \frac{\pi^2}{|\Omega|^2}k^2$ for each $k\in\cn$.
\begin{thm}\label{pertubation-product-1d}
Let $\Omega_1\subset\rr^{n_1}$ and $\Omega_2\subset\rr$ be bounded open domains, $n_1\ge 1$ and $n=n_1+1$. Suppose that
P\'olya's conjecture holds on $\Omega_1$. Then it holds for $\lambda>0$ that
\begin{align*}
\mathcal{N}^D_{\Omega_1\times\Omega_2}(\lambda)&\le  \frac{|\Omega_1||\Omega_2|\omega(n)}{(2\pi)^n}\lambda^{n/2}\left(1- \frac{C_1(n_1)}{|\Omega_2|\sqrt{\lambda}}\right)_+,
\end{align*}
where $C_1(n_1)$ is as in \eqref{defn-c1n},
\begin{align}
C_1(n_1)=\begin{cases}
\frac 23 & n_1=1,\\
\frac{3\pi}{16} & n_1=2,\\
\frac{\pi\mathscr{B}(\frac{n_1}{2}-1,2)}{6\mathscr{B}(\frac{n_1}{2}-1,5/2)} & n_1=3,4,\\
\frac{3\pi\mathscr{B}(\frac{n_1 }{2}-1,2)}{16 \mathscr{B}(\frac{n_1 }{2}-1,5/2)}\left(\frac{247}{256}\right)^{\frac{n_1-2}{2}}  &  n_1\ge 5.
\end{cases}
\end{align}
\end{thm}
\begin{proof}
 Suppose that it holds that
$$\mathcal{N}^D_{\Omega_1}(\lambda)=\#\{\lambda_k(\Omega_1):\,\lambda_k(\Omega_1)<\lambda\}\le \frac{|\Omega_1|\omega(n_1)}{(2\pi)^{n_1}} \lambda^{n_1/2}.$$
Then for all $\lambda>\lambda_1(\Omega)$, where
\begin{align}
\lambda_1(\Omega)=\lambda_1(\Omega_1)+\lambda_1(\Omega_2)\ge  \frac{(2\pi)^2}{|\Omega_1|^{2/n_1}\omega(n_1)^{2/n_1}}+ \frac{\pi^2}{|\Omega_2|^{2}},
\end{align}
we have
\begin{align*}
\mathcal{N}^D_{\Omega}(\lambda)&=\sum_{j:\,\lambda_j(\Omega_2)<\lambda }\#\left\{k:\,\lambda_k(\Omega_1)+\lambda_j(\Omega_2)<\lambda\right\}
=\sum_{j:\,\lambda_j(\Omega_2)<\lambda } \#\left\{k:\,\lambda_k(\Omega_1)<(\lambda -\lambda_j(\Omega_2))_+\right\}\\
&\le \sum_{j:\,\lambda_j(\Omega_2)<\lambda } \frac{|\Omega_1|\omega(n_1)}{(2\pi)^{n_1}} (\lambda-\lambda_j(\Omega_2))^{n_1/2}.
\end{align*}
If $n_1=1$, then by \eqref{riesz-dirichlet-n2}, we have
\begin{align*}
\mathcal{N}^D_{\Omega}(\lambda)&\le \sum_{j:\,\lambda_j(\Omega_2)<\lambda } \frac{|\Omega_1|}{\pi} (\lambda-\lambda_j(\Omega_2))^{1/2}\le  \frac{|\Omega_1||\Omega_2|}{4\pi} \lambda  \left(1- \frac{2}{3|\Omega_2|\sqrt\lambda}\right)_+.
\end{align*}
If $n_1=2$, then by \eqref{riesz-dirichlet-n1} we have
\begin{align*}
\mathcal{N}^D_{\Omega}(\lambda)&\le \sum_{j:\,\lambda_j(\Omega_2)<\lambda } \frac{|\Omega_1|\omega(2)}{(2\pi)^{2}} (\lambda-\lambda_j(\Omega_2))\le \frac{|\Omega_1|\omega(2)}{(2\pi)^{2}}\frac{2|\Omega_2|}{3\pi}\lambda^{3/2}\left(1 - \frac{3\pi}{16|\Omega_2|\sqrt\lambda} \right)_+\\
&=\frac{|\Omega|\omega(3)}{(2\pi)^{3}}\lambda^{3/2}\left(1 - \frac{3\pi}{16|\Omega_2|\sqrt\lambda} \right)_+.
\end{align*}
For  $n_1\ge 3$,  by Lemma \ref{lem-1d}, we see that
\begin{align*}
\mathcal{N}^D_{\Omega}(\lambda)&\le \sum_{j:\,\lambda_j(\Omega_2)<\lambda } \frac{|\Omega_1|\omega(n_1)}{(2\pi)^{n_1}} (\lambda-\lambda_j(\Omega_2))^{n_1/2}\\
&\le \frac{|\Omega_1|\omega(n_1)}{(2\pi)^{n_1}}\left(\frac{2|\Omega_2|}{3\pi\mathscr{B}(\frac{n_1}{2}-1,2)}  \lambda^{\frac{n_1}2+\frac{1}{2}} \mathscr{B}(\frac{n_1}{2}-1,5/2)- C_3(\frac{n_1}{2})\lambda^{\frac{n_1}{2}}\right)_+\\
&=\frac{|\Omega_1||\Omega_2|\omega(n)}{(2\pi)^{n_1+1}} \lambda^{\frac{n_1+1}{2}}\left(1-\frac{3\pi\mathscr{B}(\frac{n_1}{2}-1,2)}{2\mathscr{B}(\frac{n_1}{2}-1,5/2)}\frac{C_3(\frac{n_1}{2})}{|\Omega_2| \sqrt\lambda}\right)_+\\
&=\frac{|\Omega_1||\Omega_2|\omega(n)}{(2\pi)^{n_1+1}} \lambda^{\frac{n_1+1}{2}}\left(1-\frac{C_1(n_1)}{|\Omega_2| \sqrt\lambda}\right)_+.
 \end{align*}
The proof is complete.
\end{proof}

As a corollary,  we have the following result.
\begin{cor}\label{first-k-dirichlet-1d}
Let $\Omega_1\subset\rr^{n_1}$ and $\Omega_2\subset\rr$ be bounded open domains, $n_1\ge 1$ and $n=n_1+1$. Suppose that
P\'olya's conjecture holds on $\Omega_1$. Let $\Omega$ be an open set contained in $\Omega_1\times\Omega_2$, and
$$k_0=\left\lfloor\frac{|\Omega|\omega(n)}{(2\pi)^n}\left[\left(1-\frac{|\Omega|}{|\Omega_1||\Omega_2|}\right)^{-1}\frac{C_1(n_1)}{|\Omega_2|}\right]^{n}\right\rfloor.$$
Then for each $0\le k\le k_0$,  the P\'olya conjecture holds on $\Omega$ for $\lambda_k(\Omega)$, i.e.,
$$k\le \frac{|\Omega|\omega(n)}{(2\pi)^n}\lambda_{k}(\Omega)^{n/2}.$$
\end{cor}
\begin{proof}
Since $\Omega\subset\Omega_1\times\Omega_2$, we may assume that $|\Omega|<|\Omega_1||\Omega_2|$, otherwise $|\Omega|=|\Omega_1||\Omega_2|$ and the P\'olya conjecture holds
for all Dirichlet eigenvalues of $\Omega$.
Fix $\lambda_0>0$ such that
$$|\Omega|= |\Omega_1||\Omega_2| \left(1-\frac{C_1(n_1)}{|\Omega_2| \sqrt\lambda_0}\right).$$
By the monotonicity of Dirichlet eigenvalues, for $\lambda<\lambda_0$ we have
\begin{align}\label{3.11}
\mathcal{N}^D_{\Omega}(\lambda)&\le \mathcal{N}^D_{\Omega_1\times\Omega_2}(\lambda)\le \frac{|\Omega_1||\Omega_2|\omega(n)}{(2\pi)^n}\lambda^{n/2}\left(1-\frac{C_1(n_1)}{|\Omega_2| \sqrt\lambda}\right)\nonumber\\
&\le \frac{|\Omega|\omega(n)}{(2\pi)^n}\lambda^{n/2}.
\end{align}

Recall that
$$k_0=\left\lfloor\frac{|\Omega|\omega(n)}{(2\pi)^n}\left[\left(1-\frac{|\Omega|}{|\Omega_1||\Omega_2|}\right)^{-1}\frac{C_1(n_1)}{|\Omega_2|}\right]^{n}\right\rfloor\le \frac{|\Omega|\omega(n)}{(2\pi)^n}\lambda_0^{n/2}.$$
Let $0\le k\le k_0$. If $\lambda_{k}(\Omega)<\lambda_0$, then \eqref{3.11} gives that for small enough $\epsilon>0$ such that $\lambda_{k}(\Omega)+\epsilon<\lambda_0$
\begin{align*}
k\le \mathcal{N}^D_{\Omega}(\lambda_{k}(\Omega)+\epsilon)&\le \frac{|\Omega|\omega(n)}{(2\pi)^n}(\lambda_{k}(\Omega)+\epsilon)^{n/2}.
\end{align*}
Letting $\epsilon\to 0$ gives that
\begin{align*}
k\le \frac{|\Omega|\omega(n)}{(2\pi)^n}\lambda_{k}(\Omega)^{n/2}.
\end{align*}
If $\lambda_{k}(\Omega)\ge \lambda_0$, then by the choose of $k_0$, we see that
\begin{align*}
k\le k_0\le \frac{|\Omega|\omega(n)}{(2\pi)^n}\lambda_0^{n/2}\le \frac{|\Omega|\omega(n)}{(2\pi)^n}\lambda_k(\Omega)^{n/2}.
\end{align*}
 The proof is complete.
 \end{proof}
Obviously the above corollary only makes sense when $k_0\ge 3$.
Let us consider $\Omega_1\subset\rr^2$, with $|\Omega_1|=100$, $\Omega_2\subset\rr$ with $|\Omega_2|=1$.
Suppose that P\'olya's conjecture holds on $\Omega_1$ for the Dirichlet eigenvalues.
In this case, $C_1(2)=3\pi/16$,
\begin{align*}
\frac{|\Omega|\omega(n)}{(2\pi)^n}\left[\left(1-\frac{|\Omega|}{|\Omega_1||\Omega_2|}\right)^{-1}\frac{C_1(n_1)}{|\Omega_2|}\right]^{n}\ge \frac{|\Omega| 9\pi}{2^{13}} \left(1-\frac{|\Omega|}{100}\right)^{-3}.
\end{align*}
For any $\Omega\subset \Omega_1\times\Omega_2$, if $|\Omega|\ge 95$, then
$$\frac{|\Omega| 9\pi}{2^{13}} \left(1-\frac{|\Omega|}{100}\right)^{-3}\ge 2623,$$
and at least the first 2623 Dirichlet eigenvalues of $\Omega$ satisfy P\'olya's conjecture. If $|\Omega|\ge 99$ (resp. $|\Omega|\ge 80$), then at least the first 341694 (resp. 34) Dirichlet eigenvalues satisfy P\'olya's conjecture.

\subsection{Strip-tiling domains}\hskip\parindent
We next use Theorem \ref{pertubation-product-1d} together with P\'olya's original method to prove Theorem \ref{main}.
\begin{proof}[Proof of Theorem \ref{main}]
Given a strip-tiling domain $\Omega$, let us use  $\Omega$ to tile $\mathbb{R}^{n-1}\times(0,w(\Omega))$ and fix a such tiling.

For large enough $R>0$, consider the rectangle
$$\mathscr{R}=(-R,R)^{n-1}\times (0,w(\Omega)):=\mathscr{R}_n\times(0,w(\Omega)).$$
Let us denote the surfaces  of $ \mathscr{R}$ which are perpendicular to the $i-th$ direction, $i=1,\cdots,n-1$,
by $\mathscr{R}_{i,j}$, $j=1,2$. Then the surface area of $\mathscr{R}_{i,j}$ equals $R^{n-2}w(\Omega)$.

Let us collect all the copies of $\Omega$ used in the above tiling that locate completely in $\mathscr{R}$,  and denote
them by $\{\Omega_j\}_{j=1}^{J_R}$. See the following Figure \ref{fig4}.
\begin{figure}[htbp]
\centering
\includegraphics[scale=0.6]{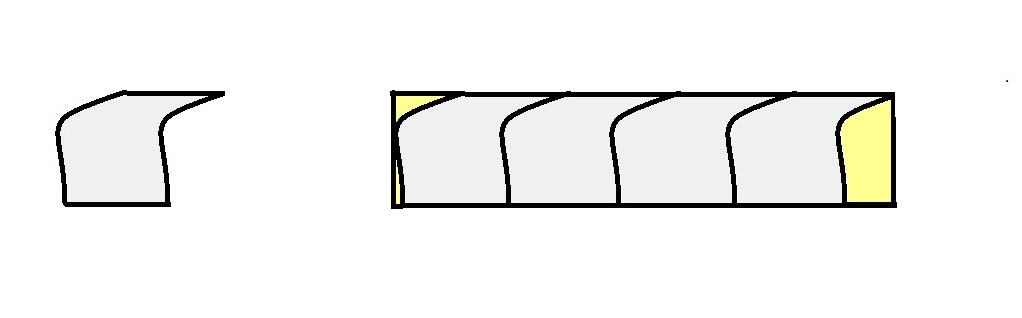}
\caption{Examples of strip-tiling domains in two dimension}
\label{fig4}
\end{figure}
For any point $x$ that belongs to $\mathscr{R}\setminus \cup_{j=1}^{J_R}\overline{\Omega_j}$, there is a copy of $\Omega$ from
the tiling, denoted by $\Omega_x$,
such that $x\in\overline{\Omega_x}$. By the previous choice,  $\Omega_x$ does not locate completely in $\mathscr{R}$,
and so
$$\Omega_x\cap \left(\mathbb{R}^{n-1}\times(0,w(\Omega)) \setminus \mathscr{R}\right)\neq \emptyset.$$
We therefore see that, there are $i\in \{1,\cdots,n-1\}$ and $j\in\{1,2\}$ such that
$$\dist(x,\mathscr{R}_{i,j})\le \mathrm{diam}(\Omega),$$
and hence,
$$\left|\mathscr{R}\setminus \cup_{j=1}^{J_R}\overline{\Omega_j}\right|\le 2(n-1)R^{n-2}w(\Omega)\mathrm{diam}(\Omega).$$
We further see that
$$J_R \ge \frac{|\mathscr{R}|-2(n-1)R^{n-2}w(\Omega)\mathrm{diam}(\Omega)}{|\Omega|}=
\frac{\left(2^{n-1}R-2(n-1)\mathrm{diam}(\Omega)\right)R^{n-2}w(\Omega)}{|\Omega|}.$$

By Theorem \ref{pertubation-product-1d},  we see that for each $k\ge 1$,
\begin{align*}
J_R N^D_{{\Omega}}(\lambda)&\le  N^D_{\cup_{j=1}^{J_R} \overline{\Omega_j}}(\lambda)\le N^D_{\mathscr{R}}(\lambda)\le  \frac{|\mathscr{R}|\omega(n)}{(2\pi)^n}\lambda^{n/2}\left(1- \frac{C_1(n-1)}{|w(\Omega)|\sqrt{\lambda}}\right)_+,
\end{align*}
which  implies that
\begin{align*}
N^D_{{\Omega}}(\lambda)&\le   \frac{\omega(n)}{(2\pi)^n} \frac{|\mathscr{R}|}{J_R}\lambda^{n/2}\left(1- \frac{C_1(n-1)}{|w(\Omega)|\sqrt{\lambda}}\right)_+\\
& \le  \frac{\omega(n)}{(2\pi)^n}  \frac{2^{n-1}R^{n-1}w(\Omega)|\Omega|}{\left(2^{n-1}R-2(n-1)\mathrm{diam}(\Omega)\right)R^{n-2}w(\Omega)}\lambda^{n/2}\left(1- \frac{C_1(n-1)}{|w(\Omega)|\sqrt{\lambda}}\right)_+
\end{align*}
Sending $R$ to $\infty$ and noting that $|\Omega|=|S_n(\Omega)|w(\Omega)$,  we see that for each $k\ge 1$,
\begin{align*}
k&\le \frac{ |\Omega|\omega(n)}{(2\pi)^n}\lambda_k(\Omega)^{n/2}-C_1(n-1)|S_n(\Omega)|\frac{\omega(n)}{(2\pi)^n}\lambda_k(\Omega)^{\frac{n-1}{2}}.
\end{align*}
The proof is complete.
\end{proof}

\begin{rem}\label{product-23} \rm
In Theorem \ref{pertubation-product-1d} and Theorem \ref{main}, we only keep one lower order term.
One can have a better estimate especially for large eigenvalues, by keeping more lower order term.
 We present it below in two and three dimensions.

Let $\Omega_1\subset\rr^{n_1}$ and $\Omega_2\subset\rr$ be bounded open domains, $n_1=1,2$. If $n_1=1$, then by \eqref{riesz-dirichlet-n2}, we have
\begin{align*}
\mathcal{N}^D_{\Omega}(\lambda)&\le \sum_{j:\,\lambda_j(\Omega_2)<\lambda } \frac{|\Omega_1|}{\pi} (\lambda-\lambda_j(\Omega_2))^{1/2}
\le  \frac{|\Omega_1|}{\pi}\left(\frac{|\Omega_2|}{4}\lambda- \frac { \sqrt{\lambda}}{2 }+
  \frac{\sqrt 6\pi}{9|\Omega_2|} \sqrt{\lfloor\frac{|\Omega_2|}{\pi}\sqrt{\lambda}\rfloor+\frac 12}\right)_+\\
  &\le \frac{|\Omega_1||\Omega_2|}{4\pi}\lambda\left(1-\frac{2}{|\Omega_2|\sqrt\lambda}+\frac{4\sqrt 6\pi}{9|\Omega_2|^2\lambda} \sqrt{\lfloor\frac{|\Omega_2|}{\pi}\sqrt{\lambda}\rfloor+\frac 12}\right)_+.
\end{align*}
If $n_1=2$, we assume that
P\'olya's conjecture holds on $\Omega_1$.
 By \eqref{riesz-dirichlet-n1} we have
\begin{align*}
\mathcal{N}^D_{\Omega}(\lambda)&\le \sum_{j:\,\lambda_j(\Omega_2)<\lambda } \frac{|\Omega_1|\omega(2)}{(2\pi)^{2}} (\lambda-\lambda_j(\Omega_2))\\
&\le \frac{|\Omega_1||\Omega_2|}{6\pi^2} \lambda^{3/2} \left(1 - \left(\frac{\lfloor\frac{|\Omega_2|\sqrt{\lambda}}{\pi}\rfloor}{\frac{|\Omega_2|\sqrt{\lambda}}{\pi}}\right)^2\frac{3\pi}{4|\Omega_2|\sqrt\lambda} \right)_+.
\end{align*}

Let $\Omega$ be a strip-tiling domain that can tile $\rr^{n-1}\times (0,w(\Omega))$, $n=2,3$, with $S_n(\Omega)$
be a face in the $n$-direction.
From the proof of Theorem \ref{main} we deduce similar estimate as following.
For the case $n=2$, it holds that
\begin{align*}
\mathcal{N}^D_{\Omega}(\lambda)&\le \frac{|\Omega|}{4\pi}\lambda\left(1-\frac{2}{|w(\Omega)|\sqrt\lambda}+\frac{4\sqrt 6\pi}{9|w(\Omega)|^2\lambda} \sqrt{\lfloor\frac{|w(\Omega)|}{\pi}\sqrt{\lambda}\rfloor+\frac 12}\right)_+,
\end{align*}
and for $n=3$, it holds that
\begin{align*}
\mathcal{N}^D_{\Omega}(\lambda)&\le
\frac{|\Omega|}{6\pi^2} \lambda^{3/2} \left(1 - \left(\frac{\lfloor\frac{|w(\Omega)|\sqrt{\lambda}}{\pi}\rfloor}{\frac{|w(\Omega)|\sqrt{\lambda}}{\pi}}\right)^2\frac{3\pi}{4|w(\Omega)|\sqrt\lambda} \right)_+.
\end{align*}
\end{rem}
Similar to Corollary \ref{first-k-dirichlet-1d} we have the following perturbation result.
\begin{cor}\label{first-k-dirichlet-strip}
Let $\Omega$ be a strip-tiling domain that can tile $\rr^{n-1}\times (0,w(\Omega))$, $n\ge 2$. Let $S_n(\Omega)$ be a side face of $\Omega$ in the $n$-th direction.
Let $\Omega_0$ be an open set contained in $\Omega$, and
$$k_0=\left\lfloor\frac{|\Omega_0|\omega(n)}{(2\pi)^n}\left[\left(1-\frac{|\Omega_0|}{|\Omega|}\right)^{-1}\frac{C_1(n-1)}{|w(\Omega)|}\right]^{n}\right\rfloor.$$
Then for each $0\le k\le k_0$,  the P\'olya conjecture holds on $\Omega_0$ for $\lambda_k(\Omega_0)$, i.e.,
$$k\le \frac{|\Omega_0|\omega(n)}{(2\pi)^n}\lambda_{k}(\Omega_0)^{n/2}.$$
\end{cor}

\section{P\'olya's conjecture on a class of domains}\hskip\parindent
Let us start with proving the  qualitative result regarding removing general smooth set.
Theorem \ref{qualitative-polya-rectangle} follows immediately from the following result.

\begin{thm}\label{qualitative-polya}
(i) Suppose that $\Omega$  is a strip-tiling domain  in $\rn$, $n\ge 2$.
Let $S_n(\Omega)$ be a side face of $\Omega$ in the $n$-th direction.
Suppose that $\Omega_3$ is a smooth domain contained in $\Omega$.
There exists $C_P=C(\Omega_3,n)$ such that if $|S_n(\Omega)|\ge C_P$ then the P\'olya conjecture holds on $\Omega\setminus\overline{\Omega_3}$.

(ii) Suppose that $\Omega_1\subset\rr^{n-1}$ satisfies P\'olya's conjecture, $n\ge 2$, $\Omega_2$ is an interval in $\rr$ and
$\Omega_3$ is a smooth domain contained in $\Omega_1\times\Omega_2$.
There exists $C_P=C(\Omega_3,n)$ such that if $|\Omega_1|\ge C_P$ then the P\'olya conjecture holds on $\Omega_1\times\Omega_2\setminus\overline{\Omega_3}$.
\end{thm}
\begin{rem}\rm
The requirement of $\Omega_3$ being smooth can be slightly weakened to
being smooth except having edges or corners, see \cite[p.871]{Seeley80}.
\end{rem}
\begin{proof}[Proof of Theorem \ref{qualitative-polya}] We only prove (ii),  the proof of (i) is the same by using Theorem \ref{main}.

Let $\Omega=\Omega_1\times\Omega_2\setminus\overline{\Omega_3}$.
Up to a scaling, we may assume that $|\Omega_3|=1$. It follows from Seeley \cite{Seeley78,Seeley80} (see also \cite[Corollary 17.5.11 \& (17.5.16)']{Hormander3}) that
for some $\Lambda>0$, it holds for $\lambda>\Lambda$ that
\begin{align}\label{seeley}
\left|\mathcal{N}^D_{\Omega_3}(\lambda)- \frac{|\Omega_3|\omega(n)}{(2\pi)^{n}} \lambda^{n/2}\right|\le C(\Omega_3)\lambda^{\frac{n-1}{2}}.
\end{align}
We assume that
$$|\Omega_1|> \max\left\{\frac{C(\Omega_3)(2\pi)^n}{\omega(n)C_1(n-1)},\frac{\sqrt{2\Lambda}}{C_1(n-1)}\right\}. $$

By Corollary \ref{first-k-dirichlet-1d}, P\'olya's conjecture holds for the first $k_0$ Dirichlet eigenvalues, where
$$k_0=\left\lfloor \frac{|\Omega|\omega(n)}{(2\pi)^n}\left[\left(1-\frac{|\Omega|}{|\Omega_1||\Omega_2|}\right)^{-1}\frac{C_1(n-1)}{|\Omega_2|}\right]^{n}\right\rfloor=
\left\lfloor \frac{|\Omega|\omega(n)}{(2\pi)^n}(C_1(n-1)|\Omega_1|)^{n}\right\rfloor.$$
It holds
$$\lambda_{k_0}(\Omega)^{n/2}\ge \frac{(2\pi)^n} {|\Omega|\omega(n)}\left\lfloor \frac{|\Omega|\omega(n)}{(2\pi)^n}(C_1(n-1)|\Omega_1|)^{n}\right\rfloor\ge \frac 12 (C_1(n-1)|\Omega_1|)^{n},$$
and hence,
$$\lambda_{k_0}(\Omega)\ge  \frac 12 (C_1(n-1)|\Omega_1|)^{2}. $$

Since $|\Omega_1|>\sqrt{2\Lambda}/C_1(n-1)$, it holds that
$$\lambda_{k_0}(\Omega)>\Lambda. $$
By \eqref{seeley} and $|\Omega_3|=1$, we see that for any $\lambda_k(\Omega)\ge \lambda_{k_0}(\Omega)$,
\begin{align*}
\mathcal{N}^D_{\Omega_3}(\lambda_k(\Omega))\ge  \frac{\omega(n)}{(2\pi)^{n}} \lambda_k(\Omega)^{n/2}- C(\Omega_3)\lambda_k(\Omega)^{\frac{n-1}{2}}.
\end{align*}
This implies on $\Omega_3$, there are   $\mathcal{N}^D_{\Omega_3}(\lambda_k(\Omega))$ eigenvalues smaller than $\lambda_k(\Omega)$. Since $\Omega$ and $\Omega_3$ are disjoint,
$\lambda_k(\Omega)$ corresponds to some $\lambda_{k+\tilde k}(\Omega\cup\Omega_3)$ for some
$$\tilde k\ge \mathcal{N}^D_{\Omega_3}(\lambda_k(\Omega))\ge \frac{ \omega(n)}{(2\pi)^{n}} \lambda_k(\Omega)^{n/2}- C(\Omega_3)\lambda_k(\Omega)^{\frac{n-1}{2}}.$$

Since $\Omega\cup\Omega_3\subset\Omega_1\times\Omega_2$ and $|\Omega\cup\Omega_3|=|\Omega_1\times\Omega_2|$, the monotonicity of Dirichlet eigenvalues together with
Theorem \ref{pertubation-product-1d} implies that
\begin{align*}
\mathcal{N}^D_{\Omega\cup \Omega_3}(\lambda)\le \mathcal{N}^D_{\Omega_1\times\Omega_2}(\lambda)&\le  \frac{|\Omega_1||\Omega_2|\omega(n)}{(2\pi)^n}\lambda^{n/2}\left(1- \frac{C_1(n-1)}{|\Omega_2|\sqrt{\lambda}}\right)_+,
\end{align*}
and therefore
\begin{align*}
 \frac{|\Omega_1||\Omega_2|\omega(n)}{(2\pi)^n}\lambda_k(\Omega)^{n/2}& =\frac{|\Omega_1||\Omega_2|\omega(n)}{(2\pi)^n}\lambda_{k+\tilde k}(\Omega\cup\Omega_3)^{n/2}\\
 &\ge k+\tilde k+ \frac{C_1(n-1)}{|\Omega_2|\sqrt{\lambda}}  \frac{|\Omega_1||\Omega_2|\omega(n)}{(2\pi)^n}\lambda_k(\Omega)^{\frac{n-1}{2}}\\
 &\ge k+\frac{ \omega(n)}{(2\pi)^{n}} \lambda_k(\Omega)^{n/2}- C(\Omega_3)\lambda_k(\Omega)^{\frac{n-1}{2}}+  \frac{C_1(n-1)|\Omega_1|\omega(n)}{(2\pi)^n}\lambda_k(\Omega)^{\frac{n-1}{2}}.
\end{align*}
Since $|\Omega_1|\ge \frac{C(\Omega_3)(2\pi)^n}{\omega(n)C_1(n-1)}$, we finally see that
\begin{align*}
 \frac{|\Omega|\omega(n)}{(2\pi)^n}\lambda_k(\Omega)^{n/2}& =\frac{|\Omega_1||\Omega_2|\omega(n)}{(2\pi)^n}\lambda_k(\Omega)^{n/2}-\frac{\omega(n)}{(2\pi)^n}\lambda_k(\Omega)^{n/2} \\
 &\ge k+  \lambda_k(\Omega)^{\frac{n-1}{2}}\left(\frac{C_1(n-1)|\Omega_1|\omega(n)}{(2\pi)^n}- C(\Omega_3)\right)\\
 &\ge k.
\end{align*}
The proof is complete.
\end{proof}
\begin{rem}\rm
Note that \eqref{seeley} can be stated for all $\lambda>0$ without referring to $\Lambda$, but generally with larger constant $C(\Omega_3)$.
We state in this way to show that the combined argument, that is for small eigenvalues using Corollary \ref{first-k-dirichlet-1d} while for large eigenvalues using the subtracting argument, will give better estimate, see Remark \ref{rem-better-constant} below.
\end{rem}

We next prove the concrete examples of domains obtained by removing cubes, Theorem \ref{polya-rectangle-minus}.
Let us start with the two dimensional case.

\subsection{The case of two dimension}
\hskip\parindent
Let $Q\subset\rr^n$ be an open unit cube, $n=2$.  Then the Dirichlet eigenvalues of $Q$ has the form
$$\left\{\lambda_k:\, \lambda_k=\pi^2(a^2+b^2),\,a,b\in\mathbb{N}\right\}.$$
Note that for a given $\lambda_k$,
$\mathcal{N}_Q^D(\lambda_k)$ is dominated by the number of unit cubes contained in the first quadrant intersecting the ball
$B(0, \sqrt{\lambda_k}/\pi)$.
So it is readily to see that for all $k\in\cn$ that
$ \lambda_k \ge 4\pi k.$
In fact Theorem \ref{pertubation-product-1d} gives a better result.
We however shall need the opposite estimates.

\begin{lem}\label{opposite-cube}
 Let $Q\subset\rr^2$ be a cube with side length $\ell(Q)$ in $\rr^2$.

(i) For any $\lambda>2\pi^2/\ell(Q)^2$, it holds that
$$\mathcal{N}_Q^D(\lambda)\ge    \frac{\ell(Q)^2}{4\pi}\lambda- \frac{\omega(2)2^{3/2}\pi\ell(Q)}{(2\pi)^2}\sqrt\lambda +\frac{\pi}{2}.$$

(ii) For any $\lambda>0$,  it holds that
$$\mathcal{N}_Q^D(\lambda)\ge    \frac{\ell(Q)^2}{4\pi}\lambda- \frac{\omega(2)2^{3/2}\pi\ell(Q)}{(2\pi)^2}\sqrt\lambda.$$
 \end{lem}
\begin{proof}
We have
\begin{align*}
\mathcal{N}_Q^D(\lambda)&=\#\left\{\lambda_k<\lambda:\, \lambda_k=\frac{\pi^2}{\ell(Q)^2}(a^2+b^2),\,a,b\in\mathbb{N}\right\}.
\end{align*}
So the number $\mathcal{N}_Q^D(\lambda)$ is not less than the number of the unit cubes contained
in the first quadrant intersecting the ball
$B(0, \ell(Q)\sqrt{\lambda}/\pi-\sqrt 2)$, where $\sqrt 2$ is the diameter of the unit cube.
This implies that for any $\lambda>2\pi^2/\ell(Q)^2$,
\begin{align*}
\mathcal{N}_Q^D(\lambda)&\ge \frac 14 |B(0,\ell(Q)\sqrt{\lambda}/\pi-\sqrt 2 )| =\frac{\pi}{4} \left(\ell(Q)\sqrt{\lambda}/\pi-\sqrt 2 \right)^2\\
&\ge \frac{\ell(Q)^2}{4\pi}\lambda- \frac{\sqrt 2\ell(Q)}{2}\sqrt\lambda +\frac{\pi}{2}.
\end{align*}

Noting that for $\lambda\le 2\pi^2/\ell(Q)^2$,
$$\frac{\ell(Q)^2}{4\pi}\lambda- \frac{\sqrt 2\ell(Q)}{2}\sqrt\lambda\le 0\le \mathcal{N}_Q^D(\lambda),$$
the conclusion is obvious. The proof is complete.
\end{proof}
\begin{rem}\label{rect-cuboid}\rm
Using \eqref{riesz-dirichlet-n2-lower} one can give a lower bound for $\mathcal{N}^D_{\mathscr{R}}(\lambda)$ for any rectangle $\mathscr{R}$ in $\rr^2$.
Inductively, a lower bound for the eigenvalue counting function can be explicitly given on any cuboid in higher dimension.
\end{rem}

Let us prove the planer case from Theorem \ref{polya-rectangle-minus}.
\begin{proof}[Proof of the case $n=2$ in Theorem \ref{polya-rectangle-minus}]
Let $\Omega$ be a strip-tiling domain that can tile $\rr\times (0,w(\Omega))$, and let $S_n(\Omega)$ be the surface  of a side face of $\Omega$ in the $n$-th direction.
Denote by $\Omega_0=\Omega\setminus\overline{\cup_{k=1}^\infty Q_{k}}$.
Let
 $$\mathscr{V}_Q:=\left|\cup_{k=1}^\infty Q_{k}\right|=\sum_{k=1}^\infty \ell(Q_k)^n.$$
Since  the P\'olya conjecture is scaling invariant, in Definition \ref{defn-cube}, we have assumed that
the largest cube in $\{Q_k\}$ has side length one.  We further have that $w(\Omega)\ge 1$ and
 $$\mathscr{V}_Q\le \mathscr{S}_Q<\infty.$$

By Theorem \ref{main} we have for $\lambda>0$ that
\begin{align*}
\mathcal{N}^D_{\Omega}(\lambda)&\le  \frac{|\Omega| }{4\pi}\lambda \left(1- \frac{{2}}{3|w(\Omega)|\sqrt{\lambda}}\right)_+.
\end{align*}
By Corollary \ref{first-k-dirichlet-strip} we have for
$$k_0=\left\lfloor\frac{|\Omega_0| }{4\pi}\left[\left(1-\frac{|\Omega_0|}{|\Omega|}\right)^{-1}\frac{2}{3|w(\Omega)|}\right]^{2}\right\rfloor= \left\lfloor\frac{|\Omega_0|}{4\pi}\left[ \frac{2|S_n(\Omega)|}{3\mathscr{V}_Q}\right]^{2}\right\rfloor,$$
the first $k_0$ Dirichlet eigenvalues of $\Omega_0$ satisfy P\'olya's conjecture, that is, for $1\le i\le k_0$,
$$\frac{|\Omega_0|\omega(2)}{(2\pi)^2}\lambda_i(\Omega_0)\ge i.$$
In particular, since $|S_n(\Omega)|\ge 2\sqrt 2 \pi \mathscr{S}_Q\ge 2\sqrt 2 \pi \mathscr{V}_Q $,
$$|\Omega_0|=|\Omega|-\mathscr{V}_Q\ge  (2\sqrt 2 \pi-1) \mathscr{V}_Q\ge (2\sqrt 2 \pi-1), $$
we find that
$$\lambda_{k_0}(\Omega_0)\ge \frac{4\pi}{|\Omega_0|}\left(\frac{|\Omega_0|}{4\pi}\left[ \frac{2|S_n(\Omega)|}{3\mathscr{V}_Q}\right]^{2}-1\right)>33. $$

By Remark \ref{product-23}, we have
\begin{align*}
\mathcal{N}^D_{\Omega}(\lambda)&\le \frac{|\Omega|}{4\pi}\lambda\left(1-\frac{2}{|w(\Omega)|\sqrt\lambda}+\frac{4\sqrt 6\pi}{9|w(\Omega)|^2\lambda} \sqrt{\lfloor\frac{|w(\Omega)|}{\pi}\sqrt{\lambda}\rfloor+\frac 12}\right)_+.
\end{align*}
This together with $|w(\Omega)|\ge 1$ implies that  for $\lambda>33$, it holds
\begin{align}\label{improved-2d-product}
 \mathcal{N}^D_{\Omega}(\lambda)&\le   \frac{|\Omega|}{4\pi}\lambda\left(1- \frac{1}{|w(\Omega)| \sqrt{\lambda}}\right).
\end{align}

Now consider the union of open sets $\Omega_0$, $Q_{k}$, $k\in\cn$,  which are disjoint and contained in $\Omega$, but
satisfy
$$|\Omega|=|\Omega_0\cup\cup_{k\in\cn} Q_{k}|.$$
So by monotonicity of Dirichlet eigenvalues, P\'olya's conjecture holds on $\Omega_0\cup\cup_{1\le k<\infty}Q_{k}$, and
\begin{align*}
\mathcal{N}^D_{\Omega_0\cup\cup_{k\in\cn}Q_{k}}(\lambda)\le \mathcal{N}^D_{\Omega}(\lambda).
\end{align*}
Moreover, by Lemma \ref{opposite-cube} (ii), we see that for $\lambda>0$,
\begin{align*}
\mathcal{N}^D_{\Omega_0\cup\cup_{k\in\cn}  Q_{k}}(\lambda)&= \mathcal{N}^D_{\Omega_0}(\lambda)+\sum_{k\in\cn}  \mathcal{N}^D_{Q_{k}}(\lambda)\\
&\ge \mathcal{N}^D_{\Omega_0}(\lambda)+\sum_{k\in\cn} \left( \frac{\ell(Q_k)^2}{4\pi}\lambda- \frac{\sqrt 2 \ell(Q_k)}{2}\sqrt\lambda\right).
\end{align*}
Combining the above two estimates, we deduce that for $\lambda\ge \lambda_{k_0}(\Omega_0)>33$,
\begin{align*}
\mathcal{N}^D_{\Omega_0}(\lambda)&\le  \frac{|\Omega|}{4\pi}\lambda\left(1- \frac{1}{|w(\Omega)| \sqrt{\lambda}}\right)-\sum_{k\in\cn} \left( \frac{\ell(Q_k)^2}{4\pi}\lambda- \frac{\sqrt 2 \ell(Q_k)}{2}\sqrt\lambda\right)\\
&=\frac{|\Omega|-\mathscr{V}_Q }{4\pi}\lambda  -\sqrt{\lambda} \left(\frac{|S_n(\Omega)|}{4\pi}-\frac{\sqrt 2}{2}\mathscr{S}_Q\right)\\
&=\frac{|\Omega_0| }{4\pi}\lambda -\sqrt{\lambda} \left(\frac{|S_n(\Omega)|}{4\pi}-\frac{\sqrt 2}{2}\mathscr{S}_Q\right).
\end{align*}
Since $|S_n(\Omega)|\ge 2\sqrt 2 \pi \mathscr{S}_Q $, we see that
\begin{align*}
\mathcal{N}^D_{\Omega_0}(\lambda)&\le \frac{|\Omega_0| }{4\pi}\lambda,
\end{align*}
which completes the proof.
\end{proof}
\begin{rem}\label{rem-better-constant}\rm
Note that in Theorem \ref{polya-rectangle-minus}, for $n\ge 4$,
$$C_2(n-1)=\frac{n^{3/2}\pi}{C_1(n-1)},$$
while for $n=2,3$
$$
\frac{n^{3/2}\pi}{C_1(n-1)}>C_2(n-1)
=\begin{cases}
2\sqrt2 \pi, &n=2,\\
9\sqrt 3, & n=3.
\end{cases}
$$
This is because in the proof of two and three dimensional case, we use Corollary \ref{first-k-dirichlet-1d}, which allows us
to reduce the problem for large eigenvalues, in which cases by using Lemma \ref{lem-riesz-dirichlet-n1} we have a better estimate than
the uniform estimate from Theorem \ref{pertubation-product-1d}, see \eqref{improved-2d-product} and \eqref{stronger-3d-dirichlet}.
\end{rem}
\subsection{Three and higher dimensions}
\hskip\parindent
We next provide the proof for $n\ge 3$.
Similar to Lemma \eqref{opposite-cube}, we have
\begin{lem}\label{opposite-cube-hign}
Let $Q\subset\rr^n$ be a  cube.

(i)  For $n=3$, for any $\lambda>3\pi^2/\ell(Q)^2$, it holds that
$$\mathcal{N}_Q^D(\lambda)\ge  \frac{\ell(Q)^3}{6\pi^2}\lambda^{3/2} - \frac{\omega(3)3^{3/2}\pi \ell(Q)^{2}}{(2\pi)^3}\lambda+1,$$
and for all $\lambda>0$ that
$$\mathcal{N}_Q^D(\lambda)\ge  \frac{\ell(Q)^3}{6\pi^2}\lambda^{3/2} -\frac{\omega(3)3^{3/2}\pi \ell(Q)^{2}}{(2\pi)^3}\lambda.$$

(ii) For $n\ge 4$, for any $\lambda>n^3\pi^2/\ell(Q)^2$, it holds that
$$\mathcal{N}_Q^D(\lambda)\ge  \frac{\omega(n)\ell(Q)^n}{(2\pi)^n} \lambda^{n/2} -  \frac{\omega(n)n^{3/2}\pi \ell(Q)^{n-1}}{(2\pi)^n}\lambda^{\frac{n-1}{2}}+2^{n-2}\pi^{n/2},$$
and for all $\lambda>0$ that
$$\mathcal{N}_Q^D(\lambda)\ge  \frac{\omega(n)\ell(Q)^n}{(2\pi)^n} \lambda^{n/2} -  \frac{\omega(n)n^{3/2}\pi \ell(Q)^{n-1}}{(2\pi)^n}\lambda^{\frac{n-1}{2}}.$$
\end{lem}
\begin{proof}
(i) For $n=3$, we have
\begin{align*}
\mathcal{N}_Q^D(\lambda)&=\#\left\{\lambda_k<\lambda:\, \lambda_k=\frac{\pi^2}{\ell(Q)^2}(a^2+b^2+c^2),\,a,b,c\in\mathbb{N}\right\}.
\end{align*}
So the number $\mathcal{N}_Q^D(\lambda)$ is not less than the number of the unit cubes contained
in the first quadrant intersecting the ball
$B(0, \ell(Q)\sqrt{\lambda}/\pi-\sqrt 3)$, where $\sqrt 3$ is the diameter of the unit cube.
We therefore deduce that
\begin{align*}
\mathcal{N}_Q^D(\lambda)&\ge \frac 18 |B(0,\ell(Q)\sqrt{\lambda}/\pi-\sqrt 3 )| =\frac{\pi}{6}\left(\ell(Q)\sqrt{\lambda}/\pi-\sqrt 3 \right)^3\\
&\ge \frac{\ell(Q)^3}{6\pi^2}\lambda^{3/2} -\frac{3\sqrt 3 \pi\ell(Q)^2}{6} \frac{\lambda}{\pi^2}+\frac{9\pi}{6} \frac{\ell(Q)\sqrt\lambda}{\pi} -\frac{3\sqrt 3\pi}{6}\\
&>\frac{\ell(Q)^3}{6\pi^2}\lambda^{3/2} -\frac{3\sqrt 3 \ell(Q)^2}{6} \frac{\lambda}{\pi}+1,
\end{align*}
for $\sqrt\lambda>\pi\sqrt 3/\ell(Q)$. For $\sqrt\lambda\le \pi\sqrt 3/\ell(Q)$, note that
$$\frac{\ell(Q)^3}{6\pi^2}\lambda^{3/2} -\frac{3\sqrt 3 \ell(Q)^2}{6} \frac{\lambda}{\pi}\le 0\le \mathcal{N}_Q^D(\lambda),$$
which completes the proof.

(ii) For $n\ge 4$, we also have
\begin{align*}
\mathcal{N}_Q^D(\lambda)&=\#\left\{\lambda_k<\lambda:\, \lambda_k=\frac{\pi^2}{\ell(Q)^2} \sum_{i=1}^n a_i^2,\,a_i \in\mathbb{N}\right\}.
\end{align*}
So the number $\mathcal{N}_Q^D(\lambda)$ is not less than the number of the unit cubes contained
in the first quadrant intersecting the ball
$B(0, \ell(Q)\sqrt{\lambda}/\pi-\sqrt n)$, where $\sqrt n$ is the diameter of the unit cube.
We therefore deduce that
\begin{align*}
\mathcal{N}_Q^D(\lambda)&\ge \frac 1{2^n} |B(0,\ell(Q)\sqrt{\lambda}/\pi-\sqrt n)| =2^{-n}\omega(n) \left(\ell(Q)\sqrt{\lambda}/\pi-\sqrt n \right)^n\\
&\ge 2^{-n}\omega(n) \sum_{i=0}^n C_n^i(\ell(Q)\lambda^{1/2}\pi^{-1})^{n-i}(-\sqrt n)^i\\
&>\frac{\omega(n)\ell(Q)^n}{(2\pi)^n} \lambda^{n/2} -  \frac{\omega(n)n^{3/2}\pi \ell(Q)^{n-1}}{(2\pi)^n}\lambda^{\frac{n-1}{2}}+2^{-n}\omega(n) \sum_{i=2}^n C_n^i(\ell(Q)\lambda^{1/2}\pi^{-1})^{n-i}(-\sqrt n)^i,
\end{align*}
where $C_n^i$ denotes the combinatoric number.
If $n\ge 4$ is odd, then by letting $k_0=\frac{n-1}{2}$, we have
\begin{align*}
& \sum_{i=2}^n C_n^i(\ell(Q)\lambda^{1/2}\pi^{-1})^{n-i}(-\sqrt n)^i\\
&\quad =\sum_{k=1}^{k_0}\left( C_n^{2k}(\ell(Q)\lambda^{1/2}\pi^{-1})^{n-2k}(\sqrt n)^{2k}- C_n^{2k+1}(\ell(Q)\lambda^{1/2}\pi^{-1})^{n-2k-1}(\sqrt n)^{2k+1}\right)=:\sum_{k=1}^{k_0}D_k.
\end{align*}
For each $1\le k\le k_0$,
\begin{align*}
  D_k&= (\ell(Q)\lambda^{1/2}\pi^{-1})^{n-2k-1}(\sqrt n)^{2k}\left(\frac{n!}{(n-2k)!(2k)!}\ell(Q)\lambda^{1/2}\pi^{-1}-\frac{n!}{(n-2k-1)!(2k+1)!}\sqrt n\right)\\
  &=\frac{n!(\ell(Q)\lambda^{1/2}\pi^{-1})^{n-2k-1}(\sqrt n)^{2k}}{(n-2k-1)!(2k)!}\left(\frac{1}{(n-2k)}\ell(Q)\lambda^{1/2}\pi^{-1}-\frac{1}{2k+1}\sqrt n\right).
\end{align*}
So if
$$\ell(Q)\lambda^{1/2}\pi^{-1} \ge n^{3/2},$$
then we have
\begin{align*}
  D_k&>\frac{n!(\ell(Q)\lambda^{1/2}\pi^{-1})^{n-2k-1}(\sqrt n)^{2k}}{(n-2k-1)!(2k)!}\left(\frac{2\ell(Q)\lambda^{1/2}\pi^{-1}}{3(n-2k)}+\frac{n^{3/2}}{3(n-2k)}-\frac{1}{2k+1}\sqrt n\right)\\
  &> \frac{2n!(\ell(Q)\lambda^{1/2}\pi^{-1})^{n-2k}(\sqrt n)^{2k}}{3(n-2k)!(2k)!}>0.
\end{align*}
Inserting these estimates into $\mathcal{N}_Q^D(\lambda)$, using the estimate of $D_1$ and positivity of $D_k$ for $k\ge 2$, we see that
\begin{align*}
\mathcal{N}_Q^D(\lambda)&\ge \frac{\omega(n)\ell(Q)^n}{(2\pi)^n} \lambda^{n/2} -  \frac{\omega(n)n^{3/2}\pi \ell(Q)^{n-1}}{(2\pi)^n}\lambda^{\frac{n-1}{2}}+2^{-n}\omega(n) \sum_{k=1}^{k_0}D_k\\
&\ge \frac{\omega(n)\ell(Q)^n}{(2\pi)^n} \lambda^{n/2} -  \frac{\omega(n)n^{3/2}\pi \ell(Q)^{n-1}}{(2\pi)^n}\lambda^{\frac{n-1}{2}}+\frac{\omega(n) n(n-1)}{3\times 2^n} (\ell(Q)\lambda^{1/2}\pi^{-1})^{n-2} n.
\end{align*}

For $n\ge 4$ being even, by dropping the last term in the binomial-sum, which is positive, we see that the same argument as above yields that
the last estimate holds for this case, too.

Since
$$\Gamma(1+\frac{n}{2}) \le \left(\frac{n+2}{2}\right)^{n/2},$$
for $\lambda$ such that $\ell(Q)\lambda^{1/2}\pi^{-1} \ge n^{3/2},$ we deduce for $n\ge 4$ that
\begin{align*}
\frac{\omega(n)}{2^{n}} \frac{n^2(n-1)}{3} (\ell(Q)\lambda^{1/2}\pi^{-1})^{n-2}&\ge \frac{n-1}{3}\frac{\pi^{n/2}}{2^{n} } \left(\frac{2}{n+2}\right)^{n/2}n^{\frac{3n}{2}-1}
>\frac{n-1}{3n}\frac{\pi^{n/2}}{2^{n} } \left(\frac{1}{n}\right)^{n/2}n^{\frac{3n}{2}} \\
& \ge \frac{\pi^{n/2}}{4}\frac{n^n}{2^{n} }\ge 2^{n-2}\pi^{n/2},
\end{align*}
and hence,
\begin{align*}
\mathcal{N}_Q^D(\lambda)&\ge \frac{\omega(n)\ell(Q)^n}{(2\pi)^n} \lambda^{n/2} -  \frac{\omega(n)n^{3/2}\pi \ell(Q)^{n-1}}{(2\pi)^n}\lambda^{\frac{n-1}{2}}+2^{n-2}\pi^{n/2}.
\end{align*}

For $\ell(Q)\lambda^{1/2}\pi^{-1} < n^{3/2},$ it holds that
$$\frac{\omega(n)\ell(Q)^n}{(2\pi)^n} \lambda^{n/2} -  \frac{\omega(n)n^{3/2}\pi \ell(Q)^{n-1}}{(2\pi)^n}\lambda^{\frac{n-1}{2}}<0\le \mathcal{N}_Q^D(\lambda),$$
which completes the proof.
\end{proof}

Recall that by \eqref{defn-c1n} and \eqref{defn-c3n}, it holds for $n\ge 4$ that
\begin{align}
C_1(n-1) =
\frac{3\pi\mathscr{B}(\frac{n-1 }{2}-1,2)}{2\mathscr{B}(\frac{n-1 }{2}-1,5/2)}C_3(\frac{n-1}{2})=\frac{3\pi\mathscr{B}(\frac{n-3 }{2},2)}{2\mathscr{B}(\frac{n-3 }{2},5/2)}C_3(\frac{n-1}{2}),
\end{align}
where
\begin{align*}
C_3(p)=
\begin{cases}
\frac 19, & 1<p\le 2,\\
\frac{1}{8 }\left(\frac{247}{256}\right)^{p-1}, & p>2.
\end{cases}
\end{align*}

\begin{thm}\label{3hd-polya} Let $n\ge 3$.

(i) Suppose that $\Omega_1\subset\rr^{n-1}$ satisfies P\'olya's conjecture (e.g. ball, tiling domain, or our example in $\rr^{n-1}$), $\Omega_2$ is an interval in $\rr$.
 Let  $\{Q_{k}\}_{k\in\cn}$ be of {\em Admissible class} in $\rn$ that is contained in $\Omega_1\times\Omega_2$.
If $|\Omega_1|\ge  C_2(n-1)\mathscr{S}_Q$, then P\'olya's conjecture holds on $(\Omega_1\times\Omega_2)\setminus \overline{\cup_{k\in\cn}Q_{k}}$,
where $C_2(n-1)$ is as in \eqref{defn-c2n},
\begin{align*}
C_2(n)=\begin{cases}
9\sqrt 3, & n=2,\\
\frac{6(n+1)^{3/2}\mathscr{B}(\frac{n}{2}-1,5/2)}{\mathscr{B}(\frac{n}{2}-1,2)} & n=3,4,\\
\frac{16(n+1)^{3/2} \mathscr{B}(\frac{n }{2}-1,5/2)}{3\mathscr{B}(\frac{n }{2}-1,2)}\left(\frac{256}{247}\right)^{\frac{n-2}{2}}  &  n\ge 5.
\end{cases}
\end{align*}

(ii) Let $\Omega$ be a strip-tiling domain in $\rn$, $n\ge 3$. Let $S_n(\Omega)$ be the surface  of a side face of $\Omega$ in the $n$-th direction.
 Let  $\{Q_{k}\}_{k\in\cn}$ be of {\em Admissible class} in $\rn$ that is contained in $\Omega$.
If $|S_n(\Omega)|\ge  C_2(n-1)\mathscr{S}_Q$, then P\'olya's conjecture holds on $\Omega\setminus \overline{\cup_{k\in\cn}Q_{k}}$,
\end{thm}
\begin{proof} We only prove (i), the proof of (ii) is the same by using Theorem \ref{main}.

Denote by $\Omega_0=\Omega_1\times\Omega_2$,
 $\Omega_3$  the union  ${\cup_{k\in\cn} Q_{k}}$, and $\Omega_4=\Omega_0\setminus \overline{\Omega_3}$.
Let
 $$\mathscr{V}_Q:=\left|\cup_{k=1}^\infty Q_{k }\right|=\sum_{k=1}^\infty \ell(Q_k)^n.$$
Recall that, since the P\'olya conjecture is scaling invariant, in Definition \ref{defn-cube}, we have assumed that
the largest cube in $\{Q_k\}$ has side length one.  We further see that $|\Omega_2|\ge 1$ and
 $$\mathscr{V}_Q\le \mathscr{S}_Q<\infty.$$

{\bf Step 1}.
Let us first consider the case $n=3$.
 By Theorem \ref{pertubation-product-1d} we have for $\lambda>0$ that
\begin{align*}
\mathcal{N}^D_{\Omega_0}(\lambda)&\le  \frac{|\Omega_1||\Omega_2|\omega(3)}{(2\pi)^3}\lambda^{3/2}\left(1- \frac{{3\pi}}{16|\Omega_2|\sqrt{\lambda}}\right)_+.
\end{align*}

For $\Omega_4$, we have by Corollary \ref{first-k-dirichlet-1d}, for $k_0$
being the largest integer satisfying
$$k_0\le \frac{|\Omega_4|\omega(3)}{(2\pi)^n}\left[\left(\frac{\mathscr{V}_Q}{|\Omega_1||\Omega_2|}\right)^{-1}\frac{3\pi}{16|\Omega_2|}\right]^{3}=\frac{9\pi |\Omega_1|^3 |\Omega_4|}{2^{13}\mathscr{V}_Q^3},$$
it holds  for each $1\le k\le k_0$,  the P\'olya conjecture holds on $\Omega_4$ for $\lambda_k(\Omega_4)$, i.e.,
$$k\le \frac{|\Omega_4|\omega(3)}{(2\pi)^3}\lambda_{k}(\Omega_4)^{3/2}.$$
Since $|\Omega_1|\ge 9\sqrt 3 \mathscr{S}_Q\ge 9\sqrt 3\mathscr{V}_Q$ and $|\Omega_2|\ge 1$, it holds that
$$\lambda_{k_0}(\Omega_4)\ge \left(6\pi^2 \frac{k_0}{|\Omega_4|}\right)^{2/3}>\left(6\pi^2 \frac{\frac{9\pi |\Omega_1|^3 |\Omega_4|}{2^{13}\mathscr{V}_Q^3}-1}{|\Omega_4|}\right)^{2/3}> \left(6\pi^2 \frac{\frac{8\pi |\Omega_1|^3 |\Omega_4|}{2^{13}\mathscr{V}_Q^3}}{|\Omega_4|}\right)^{2/3}>61.  $$

Recall that  Remark \ref{product-23} implies that for $\lambda>\lambda_1(\Omega_2)$,
\begin{align}
 \mathcal{N}^D_{\Omega_1\times\Omega_2}(\lambda)&\le \sum_{j:\,\lambda_j(\Omega_2)<\lambda } \frac{|\Omega_1|\omega(2)}{(2\pi)^{2}} (\lambda-\lambda_j(\Omega_2))\nonumber\\
&\le  \frac{|\Omega_1||\Omega_2|\omega(3)}{(2\pi)^3}\lambda^{3/2}\left(1-  \left(\frac{\lfloor\frac{|\Omega_2|\sqrt{\lambda}}{\pi}\rfloor}{\frac{|\Omega_2|\sqrt{\lambda}}{\pi}}\right)^2\frac{3\pi}{4|\Omega_2|\sqrt\lambda}\right).
\end{align}
Since  $|\Omega_2|\ge 1$, we have for $\lambda>61$  that
$$\left(\frac{\lfloor\frac{|\Omega_2|\sqrt{\lambda)}}{\pi}\rfloor}{\frac{|\Omega_2|\sqrt{\lambda}}{\pi}}\right)^2\ge \frac 49,$$
and hence,
\begin{align}\label{stronger-3d-dirichlet}
\mathcal{N}^D_{\Omega_0}(\lambda)= \mathcal{N}^D_{\Omega_1\times\Omega_2}(\lambda)&\le  \frac{|\Omega_1||\Omega_2|\omega(3)}{(2\pi)^3}\lambda^{3/2}\left(1- \frac{\pi}{3|\Omega_2|\sqrt\lambda}\right).
\end{align}

For
$$\lambda\ge \lambda_{k_0}(\Omega_4)>61,$$
by the monotonicity of Dirichlet eigenvalues and the fact $|\Omega_4\cup\Omega_3|=|\Omega_0|$, we then have
\begin{align*}
\mathcal{N}^D_{\Omega_4\cup\Omega_3}(\lambda)\le \mathcal{N}^D_{\Omega_0}(\lambda)&\le  \frac{|\Omega_1||\Omega_2|\omega(3)}{(2\pi)^3}\lambda^{3/2}\left(1- \frac{{\pi}}{3|\Omega_2|\sqrt{\lambda}}\right).
\end{align*}

Note that by Lemma \ref{opposite-cube-hign} (i),
\begin{align*}
\mathcal{N}^D_{\Omega_4\cup\Omega_3}(\lambda)&= \mathcal{N}^D_{\Omega_4}(\lambda)+\sum_{j\in\cn}\mathcal{N}^D_{Q_{j}}(\lambda)\\
&\ge \mathcal{N}^D_{\Omega_4}(\lambda)+\sum_{j\in\cn}\left(\frac{\ell(Q_{j})^3}{6\pi^2}\lambda^{3/2} -\frac{\sqrt 3 \ell(Q_{j})^2}{2} \frac{\lambda}{\pi}\right)\\
&=\mathcal{N}^D_{\Omega_4}(\lambda)+\frac{\mathscr{V}_Q}{6\pi^2}\lambda^{3/2}-\frac{\sqrt 3 \mathscr{S}_Q}{2} \frac{\lambda}{\pi}.
\end{align*}
This together with \eqref{stronger-3d-dirichlet} implies that
\begin{align*}
\mathcal{N}^D_{\Omega_4}(\lambda)&\le \frac{|\Omega_1||\Omega_2|\omega(3)}{(2\pi)^3}\lambda^{3/2}\left(1- \frac{\pi}{3|\Omega_2|\sqrt\lambda}\right)-\frac{\mathscr{V}_Q}{6\pi^2}\lambda^{3/2}+\frac{\sqrt 3 \mathscr{S}_Q}{2} \frac{\lambda}{\pi}\\
&=  \frac{|\Omega_1||\Omega_2|-\mathscr{V}_Q}{(2\pi)^3}\omega(3)\lambda^{3/2}-\lambda\left(\frac{|\Omega_1|}{18\pi}  -\frac{\sqrt 3 \mathscr{S}_Q}{2\pi} \right)\\
&=\frac{|\Omega_4|}{(2\pi)^3}\omega(3)\lambda^{3/2}-\lambda\left(\frac{|\Omega_1|}{18\pi}  -\frac{\sqrt 3 \mathscr{S}_Q}{2\pi} \right).
\end{align*}
As $|\Omega_1|\ge 9\sqrt 3 \mathscr{S}_Q$, we see that
\begin{align*}
\mathcal{N}^D_{\Omega_4}(\lambda)&\le \frac{|\Omega_4|}{(2\pi)^3}\omega(3)\lambda^{3/2},
\end{align*}
which completes the case $n=3$.


{\bf Step 2}. Let us consider the case $n\ge 4$. The proof is the same as above, we only need to verify some constants.

For $\Omega_0$, by Theorem \ref{pertubation-product-1d} we have for $\lambda>0$ that
\begin{align*}
\mathcal{N}^D_{\Omega_1\times\Omega_2}(\lambda)&\le  \frac{|\Omega_1||\Omega_2|\omega(n)}{(2\pi)^n}\lambda^{n/2}\left(1- \frac{C_1(n-1)}{|\Omega_2|\sqrt{\lambda}}\right)_+.
\end{align*}

Moreover, for $\Omega_4$, by Corollary \ref{first-k-dirichlet-1d}, we have for $k_0$ being the largest integer satisfying
\begin{align}
k_0 &\le \frac{|\Omega_4|\omega(n)}{(2\pi)^n}\left[\left(1-\frac{|\Omega_4|}{|\Omega_1||\Omega_2|}\right)^{-1}\frac{C_1(n-1)}{|\Omega_2|}\right]^{n}= \frac{|\Omega_4|\omega(n)}{(2\pi)^n}\left[\left(\frac{\mathscr{V}_Q}{|\Omega_1||\Omega_2|}\right)^{-1}\frac{C_1(n-1)}{|\Omega_2|}\right]^{n}\nonumber\\
&=\frac{|\Omega_4|\omega(n)}{(2\pi)^n}  \frac{(C_1(n-1)|\Omega_1|)^n}{\mathscr{V}_Q^n},
\end{align}
the P\'olya conjecture holds true for $1\le k\le k_0$ on $\Omega_4$, i.e.,
$$\frac{|\Omega_4|\omega(n)}{(2\pi)^n}\lambda_k(\Omega_4)^{n/2}\ge k. $$
So for $\lambda_{k_0}(\Omega_4)$, as $|\Omega_1|\ge n^{3/2}\pi \mathscr{S}_Q /C_1(n-1) \ge n^{3/2}\pi \mathscr{V}_Q /C_1(n-1)$, we have
\begin{align*}
\lambda_{k_0}(\Omega_4)&\ge \left(\frac{(2\pi)^n}{|\Omega_4|\omega(n)} k_0\right)^{2/n} \ge \left(\frac{(2\pi)^n}{2|\Omega_4|\omega(n)} \frac{|\Omega_4|\omega(n)}{(2\pi)^n} \frac{ (C_1(n-1)|\Omega_1|)^n }{\mathscr{V}_Q^n} \right)^{2/n} \\
&\ge \frac 12 \frac{ (C_1(n-1)|\Omega_1|)^2 }{\mathscr{V}_Q^2}>n^3\pi^2/2.
\end{align*}

For
$$\lambda\ge \lambda_{k_0}(\Omega_4),$$
by the monotonicity of Dirichlet eigenvalues and the fact $|\Omega_4\cup\Omega_3|=|\Omega_0|$, we then have
\begin{align*}
\mathcal{N}^D_{\Omega_4\cup\Omega_3}(\lambda)\le \mathcal{N}^D_{\Omega_0}(\lambda)&\le   \frac{|\Omega_1||\Omega_2|\omega(n)}{(2\pi)^n}\lambda^{n/2}\left(1- \frac{C_1(n-1)}{|\Omega_2|\sqrt{\lambda}}\right).
\end{align*}
By Lemma \ref{opposite-cube-hign} (ii),
\begin{align*}
\mathcal{N}^D_{\Omega_4\cup\Omega_3}(\lambda)&= \mathcal{N}^D_{\Omega_4}(\lambda)+\sum_{j\in\cn}\mathcal{N}^D_{Q_{j}}(\lambda)\\
&\ge \mathcal{N}^D_{\Omega_4}(\lambda)+\sum_{j\in\cn}\left(\frac{\omega(n)\ell(Q_{j})^n}{(2\pi)^n} \lambda^{n/2} -  \frac{\omega(n)n^{3/2}\pi \ell(Q_{j})^{n-1}}{(2\pi)^n}\lambda^{\frac{n-1}{2}}\right)\\
&=\mathcal{N}^D_{\Omega_4}(\lambda)+\frac{\omega(n)\mathscr{V}_Q}{(2\pi)^n} \lambda^{n/2}-\frac{\omega(n)n^{3/2}\pi \mathscr{S}_Q}{(2\pi)^n}\lambda^{\frac{n-1}{2}}.
\end{align*}
The above two inequalities  imply  that
\begin{align*}
\mathcal{N}^D_{\Omega_4}(\lambda)&\le \frac{|\Omega_1||\Omega_2|\omega(n)}{(2\pi)^n}\lambda^{n/2}\left(1- \frac{C_1(n-1)}{|\Omega_2|\sqrt{\lambda}}\right)-\left(\frac{\omega(n)\mathscr{V}_Q}{(2\pi)^n} \lambda^{n/2}-\frac{\omega(n)n^{3/2}\pi \mathscr{S}_Q}{(2\pi)^n}\lambda^{\frac{n-1}{2}}\right)\\
&=  \frac{|\Omega_1||\Omega_2|-\mathscr{V}_Q}{(2\pi)^n}\omega(n)\lambda^{n/2}-\frac{\lambda^{\frac{n-1}{2}}\omega(n)}{(2\pi)^n} \left(C_1(n-1)|\Omega_1|  -n^{3/2}\pi \mathscr{S}_Q  \right)\\
&= \frac{|\Omega_4|}{(2\pi)^n}\omega(n)\lambda^{n/2}-\frac{\lambda^{\frac{n-1}{2}}\omega(n)}{(2\pi)^n} \left(C_1(n-1)|\Omega_1|  -n^{3/2}\pi \mathscr{S}_Q  \right).
\end{align*}
As $|\Omega_1|\ge n^{3/2}\pi \mathscr{S}_Q /C_1(n-1)=C_2(n-1)\mathscr{S}_Q$, we see that
\begin{align*}
\mathcal{N}^D_{\Omega_4}(\lambda)&\le \frac{|\Omega_4|}{(2\pi)^n}\omega(n)\lambda^{n/2},
\end{align*}
which completes the proof.
\end{proof}
\begin{proof}[Proof of the case $n\ge 3$ in Theorem \ref{polya-rectangle-minus}]
The result follows from Theorem \ref{3hd-polya} immediately.
\end{proof}
\subsection{Some other class satisfying P\'olya's conjecture}
\hskip\parindent
We next provide some result regarding the case of general domains satisfying P\'olya's conjecture.
\begin{thm}\label{removing-domain-nonproduct}
Let $\Omega_1\subset\rr^{n}$,  $n\ge 2$. Suppose that
P\'olya's conjecture holds on $\Omega_1$, and $\{Q_j\}_{j\in\cn}$ is of Admissible class contained in $\Omega_1$.
Then it holds on  $\Omega=\Omega_1\setminus\overline{\cup_{j\in\cn}Q_j}$ that
\begin{align}
k\le \begin{cases}
\frac{|\Omega|\omega(n)}{(2\pi)^n}\lambda_k(\Omega)^{n/2},\, &\, k\le 2,\\
\frac{|\Omega|\omega(n)}{(2\pi)^n}\lambda_k(\Omega)^{n/2}+\frac{\omega(n)n^{3/2}\pi \mathscr{S}_Q}{(2\pi)^n}\lambda_k(\Omega)^{\frac{n-1}{2}},\,&\, k>2.
\end{cases}
\end{align}
\end{thm}
\begin{proof}
The required inequalities for $k\le 2$ followed from the Rayleigh-Faber-Krahn inequality ($\lambda_1(\Omega)$) and   the Krahn-Szeg\"o inequality  ($\lambda_2(\Omega)$), see \cite{Henrot06}.

Let us consider $k\ge 3$. Note that $\Omega\cup \cup_{j}Q_j\subsetneq \Omega_1$, and $|\Omega\cup \cup_jQ_j|=|\Omega_1|$.
By Lemma \ref{opposite-cube} and Lemma \ref{opposite-cube-hign} , and the fact $\Omega$, $Q_j$ are mutually disjoint, we find that
\begin{align*}
\mathcal{N}^D_{ \Omega_1}(\lambda)&\ge \mathcal{N}^D_{\Omega\cup \cup_{j}Q_j}(\lambda)= \mathcal{N}^D_{\Omega}(\lambda)+\sum_{j\in\cn}\mathcal{N}^D_{Q_{j}}(\lambda)\\
&\ge \mathcal{N}^D_{\Omega}(\lambda)+\sum_{j\in\cn}\left(\frac{\omega(n)\ell(Q_{j})^n}{(2\pi)^n} \lambda^{n/2} -  \frac{\omega(n)n^{3/2}\pi \ell(Q_{j})^{n-1}}{(2\pi)^n}\lambda^{\frac{n-1}{2}}\right)\\
&=\mathcal{N}^D_{\Omega}(\lambda)+\frac{\omega(n)\mathscr{V}_Q}{(2\pi)^n} \lambda^{n/2}-\frac{\omega(n)n^{3/2}\pi \mathscr{S}_Q}{(2\pi)^n}\lambda^{\frac{n-1}{2}},
\end{align*}
where
$$\mathscr{V}_Q:=\left|\cup_{k=1}^\infty Q_{k }\right|=\sum_{k=1}^\infty \ell(Q_k)^n.$$
This together with the assumption that $\Omega_1$ satisfies P\'olya's conjecture implies that
\begin{align}\label{e-product-new}
\mathcal{N}^D_{\Omega}(\lambda)&\le \frac{|\Omega_1|\omega(n)}{(2\pi)^n}\lambda^{n/2}-\left(\frac{\omega(n)\mathscr{V}_Q}{(2\pi)^n} \lambda^{n/2}-\frac{\omega(n)n^{3/2}\pi \mathscr{S}_Q}{(2\pi)^n}\lambda^{\frac{n-1}{2}}\right)\nonumber\\
&= \frac{|\Omega|\omega(n)}{(2\pi)^n}\lambda^{n/2}+\frac{\omega(n)n^{3/2}\pi \mathscr{S}_Q}{(2\pi)^n}\lambda^{\frac{n-1}{2}}.
\end{align}
This is equivalent to that, for $k\ge 3$,
\begin{align}
k\le \frac{|\Omega|}{(2\pi)^n}\omega(n)\lambda_k^{n/2}+\frac{\omega(n)n^{3/2}\pi \mathscr{S}_Q}{(2\pi)^n}\lambda_k^{\frac{n-1}{2}},
\end{align}
which completes the proof.
\end{proof}
Inspired by \cite[Theorem 1.2]{HW24}, using Theorem \ref{removing-domain-nonproduct} we have another class of domains
satisfying P\'olya's conjecture.
\begin{cor}\label{cor-almost-polya}
Let $\Omega_1\subset\rr^{n}$, $n\ge 2$. Suppose that
P\'olya's conjecture holds on $\Omega_1$ (e.g., tiling domain, ball, or our example), and $\{Q_j\}_{j\in\cn}$ is of Admissible class contained in $\Omega_1$.
Then the P\'olya conjecture holds on  $\Omega=(\Omega_1\setminus\overline{\cup_{j\in\cn}Q_j})\times \Omega_2$, whenever
$\Omega_2\subset\rr$ satisfies
$$|\Omega_2|   \le
\begin{cases}
\frac{|\Omega_1\setminus\overline{\cup_{j\in\cn}Q_j}|}{4\sqrt 2 \pi \mathscr{S}_Q }, &\, n=2,\\
\frac{|\Omega_1\setminus\overline{\cup_{j\in\cn}Q_j}|}{ 2\times 3^{5/2} \mathscr{S}_Q }\left(\frac{247}{256}\right)^{\frac{n-2}{2}}, &\, n=3,\\
\frac{|\Omega_1\setminus\overline{\cup_{j\in\cn}Q_j}|}{3 \mathscr{S}_Q }\left(\frac{247}{256}\right)^{\frac{n-2}{2}} \frac{\mathscr{B}(\frac{n-3}{2},2)} {2 n^{3/2}\mathscr{B}(\frac{n-3}{2},5/2)}, &\, n\ge 4.
\end{cases}
$$
\end{cor}
\begin{proof}
The proof is similar to the proof of Theorem \ref{pertubation-product-1d}.
Let $\Omega_0=\Omega_1\setminus\overline{\cup_{j\in\cn}Q_j}$.
Then for all $\lambda>\lambda_1(\Omega)$, where
\begin{align}
\lambda_1(\Omega)=\lambda_1(\Omega_0)+\lambda_1(\Omega_2)\ge  \frac{(2\pi)^2}{|\Omega_0|^{2/n}\omega(n)^{2/n}}+ \frac{\pi}{|\Omega_2|^{2}},
\end{align}
we have via \eqref{e-product-new} that
\begin{align*}
\mathcal{N}^D_{\Omega}(\lambda)&=\sum_{j:\,\lambda_j(\Omega_2)<\lambda }\#\left\{k:\,\lambda_k(\Omega_0)+\lambda_j(\Omega_2)<\lambda\right\}
=\sum_{j:\,\lambda_j(\Omega_2)<\lambda } \#\left\{k:\,\lambda_k(\Omega_0)<(\lambda -\lambda_j(\Omega_2))\right\}\\
&\le \sum_{j:\,\lambda_j(\Omega_2)<\lambda } \left[\frac{|\Omega_0|\omega(n)}{(2\pi)^{n}}  (\lambda -\lambda_j(\Omega_2))^{n/2}+  \frac{\omega(n)n^{3/2}\pi \mathscr{S}_Q}{(2\pi)^n}(\lambda -\lambda_j(\Omega_2))^{\frac{n-1}{2}}\right].
\end{align*}
For $n=2$, via Lemma \ref{lem-riesz-dirichlet-n1}  we find that
\begin{align*}
\mathcal{N}^D_{\Omega}(\lambda)&\le \sum_{j:\,\lambda_j(\Omega_2)<\lambda } \left[\frac{|\Omega_0|\omega(n)}{(2\pi)^{n}}  (\lambda -\lambda_j(\Omega_2))^{n/2}+  \frac{\omega(n)n^{3/2}\pi \mathscr{S}_Q}{(2\pi)^n}(\lambda -\lambda_j(\Omega_2))^{\frac{n-1}{2}}\right]\\
&\le \frac{|\Omega_0|\omega(n)}{(2\pi)^{n}}  \frac{2|\Omega_2|}{3\pi} \lambda^{3/2}  \left(1 - \frac{3\pi}{16|\Omega_2|\sqrt\lambda} \right) +\frac{\omega(n)n^{3/2}\pi \mathscr{S}_Q}{(2\pi)^n} \frac {|\Omega_2|}{4} \lambda\\
&= \frac{|\Omega_0||\Omega_2|}{6\pi^2}\lambda^{3/2}+\frac{ \omega(n)\lambda }{(2\pi)^{n}} \left({n^{3/2}\pi \mathscr{S}_Q} \frac {|\Omega_2|}{4}-\frac{|\Omega_0|}{8}\right).
\end{align*}
For
$$|\Omega_2|\le \frac{|\Omega_0|}{4\sqrt 2 \pi \mathscr{S}_Q },$$
we see that
\begin{align*}
\mathcal{N}^D_{\Omega}(\lambda)&\le  \frac{|\Omega_0||\Omega_2|}{6\pi^2}\lambda^{3/2}.
\end{align*}

For $n=3$, via Lemma \ref{lem-riesz-dirichlet-n1}  and Lemma \ref{lem-1d} we have that
\begin{align*}
\mathcal{N}^D_{\Omega}(\lambda)&\le \sum_{j:\,\lambda_j(\Omega_2)<\lambda }\left[ \frac{|\Omega_0|\omega(n)}{(2\pi)^{n}}  (\lambda -\lambda_j(\Omega_2))^{n/2}+  \frac{\omega(n)n^{3/2}\pi  \mathscr{S}_Q}{(2\pi)^n} (\lambda -\lambda_j(\Omega_2))^{\frac{n-1}{2}}\right]\\
&\le \frac{|\Omega_0|\omega(n)}{(2\pi)^{n}} \left( \frac{2|\Omega_2|}{3\pi\mathscr{B}(\frac{n-2}{2},2)}  \lambda^{\frac{n+1}{2}} \mathscr{B}(\frac{n-2}{2},5/2)-\frac{1}{9}\left(\frac{247}{256}\right)^{\frac{n-2}{2}}  \lambda^{\frac{n}{2}}  \right)\\
&\quad+ \frac{\omega(n)n^{3/2}\pi  \mathscr{S}_Q}{(2\pi)^n} \frac{2|\Omega_2|}{3\pi} \lambda^{n/2}\\
&=\frac{|\Omega_0||\Omega_2|\omega(n+1)}{(2\pi)^{n+1}} \lambda^{\frac{n+1}{2}} +\frac{ \omega(n)\lambda^{\frac n2}}{(2\pi)^{n}}\left( \frac{2n^{3/2} \mathscr{S}_Q |\Omega_2|}{3} -
\frac{|\Omega_0|}{9}\left(\frac{247}{256}\right)^{\frac{n-2}{2}}  \right).
\end{align*}
The desired conclusion follows as soon as
$$ |\Omega_2| \le
\frac{|\Omega_0|}{ 2\times 3^{5/2} \mathscr{S}_Q }\left(\frac{247}{256}\right)^{\frac{n-2}{2}} .$$

For $n\ge 4$, we deduce via Lemma \ref{lem-1d} that
\begin{align*}
\mathcal{N}^D_{\Omega}(\lambda)&\le  \sum_{j:\,\lambda_j(\Omega_2)<\lambda }\left[ \frac{|\Omega_0|\omega(n)}{(2\pi)^{n}}  (\lambda -\lambda_j(\Omega_2))^{n/2}+  \frac{\omega(n)n^{3/2}\pi  \mathscr{S}_Q}{(2\pi)^n} (\lambda -\lambda_j(\Omega_2))^{\frac{n-1}{2}}\right]\\
&\le \frac{|\Omega_0|\omega(n)}{(2\pi)^{n}} \left( \frac{2|\Omega_2|}{3\pi\mathscr{B}(\frac{n-2}{2},2)}  \lambda^{\frac{n+1}{2}} \mathscr{B}(\frac{n-2}{2},5/2)-\frac{1}{9}\left(\frac{247}{256}\right)^{\frac{n-2}{2}}  \lambda^{\frac{n}{2}}  \right)\\
&\quad+ \frac{\omega(n)n^{3/2}\pi  \mathscr{S}_Q}{(2\pi)^n} \frac{2|\Omega_2|}{3\pi\mathscr{B}(\frac{n-3}{2},2)}  \lambda^{\frac{n}{2}} \mathscr{B}(\frac{n-3}{2},5/2)\\
&=\frac{|\Omega_0||\Omega_2|\omega(n+1)}{(2\pi)^{n+1}} \lambda^{\frac{n+1}{2}} +\frac{ \omega(n)\lambda^{\frac n2}}{(2\pi)^{n}}\left( \frac{2n^{3/2} \mathscr{S}_Q |\Omega_2|\mathscr{B}(\frac{n-3}{2},5/2)}{3\mathscr{B}(\frac{n-3}{2},2)} -
\frac{|\Omega_0|}{9}\left(\frac{247}{256}\right)^{\frac{n-2}{2}}  \right).
\end{align*}
Since
$$|\Omega_2|   \le
\frac{|\Omega_0|}{3}\left(\frac{247}{256}\right)^{\frac{n-2}{2}} \frac{\mathscr{B}(\frac{n-3}{2},2)} {2n^{3/2}  \mathscr{S}_Q \mathscr{B}(\frac{n-3}{2},5/2)},$$
we see that
 \begin{align*}
\mathcal{N}^D_{\Omega}(\lambda)&\le \frac{|\Omega_0||\Omega_2|\omega(n+1)}{(2\pi)^{n+1}} \lambda^{\frac{n+1}{2}},
\end{align*}
which completes the proof.
\end{proof}

\section{Quantitative remainder estimate}
\hskip\parindent In this section, we prove our main uniform estimates on Lipschitz domains.
Let us recall the Whitney decomposition, see \cite[Appendix J]{Grafakos08}.
\begin{prop}\label{whitney}
Let $\Omega$ be an open nonempty subset of $\rn$. Then there exists a family of dyadic cubes $\{Q_i\}_j$ such that

(a) $\cup_j Q_j=\Omega$ and the $Q_j's$ have disjoint interiors.

(b) $\sqrt n \ell(Q_j)\le \mathrm{dist}(Q_j,\Omega^c)\le 4\sqrt n\ell(Q_j).$

(c) If the boundaries of $Q_j$ and $Q_k$ touch, then
$$\frac 14\le \frac{\ell(Q_j)}{\ell(Q_k)}\le 4.$$

(d) For a given $Q_j$ there exist at most $12^n$ $Q_k$'s that touch it.
\end{prop}

We first provide the proof for the estimate on the remainder of Weyl's law.
\begin{proof}[Proof of Theorem  \ref{weyl-remainder-uniform}]
{\bf Step 1. Upper bound for Dirichlet eigenvalues}.
Let $\{Q_j\}_{j\in\cn}$ be a sequence of dyadic cubes that forms a Whitney decomposition of $\Omega$ by
Proposition \ref{whitney}. Suppose that $B(x_0,r_{in})$ is the largest inner ball contained in $\Omega$.

Let $Q_l$ be a cube from $\{Q_j\}_j$ such that $x_0\in \overline{Q_l}$.
Then it holds that
$$\sqrt n\ell(Q_l)\le \mathrm{dist}(Q_l,\Omega^c)\le r_{in}\le 5\sqrt n \ell(Q_l).$$
If there is some other $Q_i$, with $\ell(Q_i)>\ell(Q_l)$, then
$$\sqrt n \ell(Q_i)\le r_{in}\le 5\sqrt n \ell(Q_l).$$
Let $2^{-k_0}$ be the side length of the largest cube in $\{Q_j\}_j$. So it holds that
$$\frac{r_{in}}{5\sqrt n}\le 2^{-k_0}=\sup_{Q_j}\{\ell(Q_j)\}\le \frac{r_{in}}{\sqrt n}.$$

For $0<\epsilon<r_{in}$, let
$$\Omega_\epsilon=\{x\in\Omega:\,\mathrm{dist}(x,\partial\Omega)<\epsilon\}.$$
Let $k_\epsilon\in\mathbb{Z}$ be such that
$$5\sqrt n 2^{-k_\epsilon}\le \epsilon< 5\sqrt n 2^{-k_\epsilon+1}.$$
Then $k_\epsilon>k_0$.

By Lemma \ref{opposite-cube} and Lemma \ref{opposite-cube-hign}, we have for any $\lambda>0$,
\begin{align*}
\mathcal{N}_\Omega^D(\lambda)&\ge \sum_{\ell(Q_j)>2^{-k_\epsilon}} \mathcal{N}_{Q_j}^D(\lambda)\\
&=\sum_{m=k_0}^{k_\epsilon-1} \sum_{Q_j:\,\ell(Q_j)=2^{-m}} \mathcal{N}_{Q_j}^D(\lambda) \\
&\ge \sum_{m=k_0}^{k_\epsilon-1} \sum_{Q_j:\,\ell(Q_j)=2^{-m}}  \left( \frac{\omega(n)\ell(Q_j)^n}{(2\pi)^n} \lambda^{n/2} -  \frac{\omega(n)n^{3/2}\pi \ell(Q_j)^{n-1}}{(2\pi)^n}\lambda^{\frac{n-1}{2}}\right).
\end{align*}
Note that all the cubes $\{Q_j\}_j$ with side length smaller than $2^{-k_\epsilon+1}$ is contained in $\Omega_\epsilon$. Moreover,
 the number of cubes of side length $2^{-j}$, $j\ge k_0$, is smaller than
$$\frac{|\Omega_{5\sqrt n 2^{-j}}|}{2^{-jn}}\le \frac{C_{Lip}(\Omega) 5\sqrt n 2^{-j}|\partial\Omega| }{2^{-jn}}\le 5\sqrt n 2^{j(n-1)}C_{Lip}(\Omega) |\partial\Omega|,$$
where we used the fact that
$$|  \Omega_{5\sqrt n 2^{-j}} |=\left|\{x\in\Omega:\,\mathrm{dist}(x,\partial\Omega)<5\sqrt n 2^{-j}\}\right|\le 5\sqrt n 2^{-j} C_{Lip}(\Omega) |\partial\Omega|,$$
$C_{Lip}(\Omega)$ is defined as in  \eqref{constant-omega}.

We therefore deduce that
\begin{align*}
\mathcal{N}_\Omega^D(\lambda)&\ge   \frac{\omega(n)|\Omega\setminus \Omega_\epsilon|}{(2\pi)^n} \lambda^{n/2} -   \sum_{m=k_0}^{k_\epsilon-1}  5\sqrt n C_{Lip}(\Omega)|\partial\Omega| \frac{\omega(n)n^{3/2}\pi  }{(2\pi)^n}\lambda^{\frac{n-1}{2}}\nonumber \\
&\ge  \frac{\omega(n)|\Omega|}{(2\pi)^n} \lambda^{n/2}-  \frac{\omega(n)|\Omega_\epsilon|}{(2\pi)^n} \lambda^{n/2} -5\sqrt n (k_\epsilon-k_0)C_{Lip}(\Omega)|\partial\Omega| \frac{\omega(n)n^{3/2}\pi  }{(2\pi)^n}\lambda^{\frac{n-1}{2}}.
\end{align*}
Note that
$$|  \Omega_\epsilon|=\left|\{x\in\Omega:\,\mathrm{dist}(x,\partial\Omega)<\epsilon\}\right|\le C_{Lip}(\Omega)\epsilon |\partial\Omega|$$
and
$$2^{k_\epsilon-k_0}\le {\frac{10\sqrt n}{\epsilon}}{\frac{r_{in}}{\sqrt n }}\le \frac{10r_{in}}{\epsilon}.$$
We therefore have
\begin{align}\label{epsilon-lower-convex}
\mathcal{N}_\Omega^D(\lambda)&\ge
 \frac{\omega(n)|\Omega|}{(2\pi)^n} \lambda^{n/2}-  \frac{C_{Lip}(\Omega)\omega(n)|\partial\Omega|}{(2\pi)^n}\lambda^{\frac{n-1}{2}}\left(\epsilon \lambda^{1/2} +5n^{2}\pi\log_2\frac{10r_{in}}{\epsilon} \right).
\end{align}
Let $\mathscr{R}=\mathscr{R}_{n-1}\times I_{min}$ be the minimal rectangle for $\Omega$. Since
$$\lambda_1(\Omega)\ge \lambda_1(\mathscr{R})>\frac{\pi^2}{|I_{min}|^2},$$
we see that
$$\frac{\pi r_{in}}{|I_{min}|\sqrt{\lambda_1(\Omega)}}<  r_{in}. $$
So for each $k\in\cn$ we can let
$$\epsilon=\frac{\pi r_{in}}{|I_{min}|\sqrt{\lambda_k(\Omega)}}$$
and deduce that
\begin{align*}
k&\ge
 \frac{\omega(n)|\Omega|}{(2\pi)^n} \lambda_k(\Omega)^{n/2}-  \frac{C_{Lip}(\Omega)\omega(n)|\partial\Omega| }{(2\pi)^n}\lambda_k(\Omega)^{\frac{n-1}{2}}\left(\frac{\pi r_{in}}{|I_{min}|}+5n^{2}\pi\log_2\frac{10|I_{min}| \sqrt{\lambda_k(\Omega)}}{\pi } \right).
\end{align*}
The proof of {\bf Step 1} is complete.

{\bf Step 2. Lower bound for the Dirichlet eigenvalues}.
To prove the lower bound  for the Dirichlet eigenvalues of $\Omega$, let us set $\Omega_1=\mathscr{R}\setminus \overline{\Omega}$. Obviously, $\Omega_1$ is also a Lipschitz domain.
Moreover, by \eqref{constant-omega}, it holds that
\begin{align}\label{constant-complement-omega}
\left|\{x\in \Omega_1:\,\mathrm{dist}(x,\partial \Omega_1)<\epsilon\}\right|
&\le \epsilon |\partial \mathscr{R}\cap\partial\Omega_1|+C_{Lip}(\Omega)\epsilon |\partial\Omega_1 \cap(\partial\Omega\setminus\partial\mathscr{R})|\nonumber\\
&\le C_{Lip}(\Omega) \epsilon|\partial\Omega_1|.
\end{align}
Recall also that $r_{in}(\Omega_1)$ denotes the inner radius of $\Omega_1$.

Since $\Omega\cup \Omega_1\subset\mathscr{R}=\mathscr{R}_{n-1}\times I_{min}$ and $|\Omega\cup \Omega_1|=|\mathscr{R}|$, we have via the monotonicity of Dirichlet eigenvalues  for any $\lambda>0$ that
\begin{align*}
\mathcal{N}_{\Omega\cup\Omega_1}^D(\lambda)=\mathcal{N}_{\Omega}^D(\lambda)+\mathcal{N}_{\Omega_1}^D(\lambda)\le \mathcal{N}_{\mathscr{R}}^D(\lambda).
\end{align*}
Applying {\bf Step 1} to $\Omega_1$, \eqref{constant-complement-omega} and Theorem \ref{pertubation-product-1d}, we find that
\begin{align*}
\mathcal{N}_{\Omega}^D(\lambda)&\le  \mathcal{N}_{\mathscr{R}}^D(\lambda)-\mathcal{N}_{\Omega_1}^D(\lambda)\\
& \le \frac{|\mathscr{R}_{n-1}||I_{min}|\omega(n)}{(2\pi)^n}\lambda^{n/2}\left(1- \frac{C_1(n-1)}{|I_{min}|\sqrt{\lambda}}\right)_+\\
&\quad - \frac{\omega(n)|\Omega_1|}{(2\pi)^n} \lambda^{n/2}+  \frac{C_{Lip}(\Omega)\omega(n)|\partial\Omega_1|}{(2\pi)^n}\lambda^{\frac{n-1}{2}}\left(\epsilon \lambda^{1/2} +5n^{2}\pi\log_2\frac{10r_{in}(\Omega_1) }{\epsilon} \right),
\end{align*}
for $0<\epsilon<r_{in}(\Omega_1)$.

Since
$$\lambda_1(\Omega)\ge \lambda_1(\mathscr{R})>\frac{\pi^2}{|I_{min}|^2},$$
we see that
$$\frac{\pi r_{in}(\Omega_1)}{|I_{min}|\sqrt{\lambda_1(\Omega)}}<  r_{in}(\Omega_1). $$
So for each $k\in\cn$ we can let
$$\epsilon=\frac{\pi r_{in}(\Omega_1)}{|I_{min}|\sqrt{\lambda_k(\Omega)}}$$
and get that
\begin{align*}
k&\le \frac{|\Omega|\omega(n)}{(2\pi)^n}\lambda_k(\Omega)^{n/2}\\
&\quad+  \frac{\omega(n)}{(2\pi)^n}\lambda_k(\Omega)^{\frac{n-1}{2}} \left(C_{Lip}(\Omega)|\partial\Omega_1| \left(\frac{\pi r_{in}(\Omega_1)}{|I_{min}|}+5n^{2}\pi\log_2\frac{10|I_{min}|\sqrt{\lambda_k(\Omega)}}{\pi} \right)-C_1(n-1)|\mathscr{R}_{n-1}|\right).
\end{align*}
The proof is complete.
\end{proof}

We next sketch the proof of Corollary \ref{cor-uniform-weyl}.
\begin{proof}[Proof of Corollary \ref{cor-uniform-weyl}]
If $\Omega\subset\rn$ is a convex domain, then obviously
$$|  \Omega_\epsilon|=\left|\{x\in\Omega:\,\mathrm{dist}(x,\partial\Omega)<\epsilon\}\right|\le \epsilon |\partial\Omega|,$$
i.e., $C_{Lip}(\Omega)=1$ in {\bf Step 1} above.

For {\bf Step 2}, when dealing with the complementary set $\mathscr{R}\setminus\overline{\Omega}$,
one can carry out a Whitney decomposition of the domain $\mathscr{R}\setminus \overline{\Omega}$, w.r.t. to the boundary
$\partial\Omega$, see Figure \ref{fig5}.
\begin{figure}[htbp]
\centering
\includegraphics[scale=0.5]{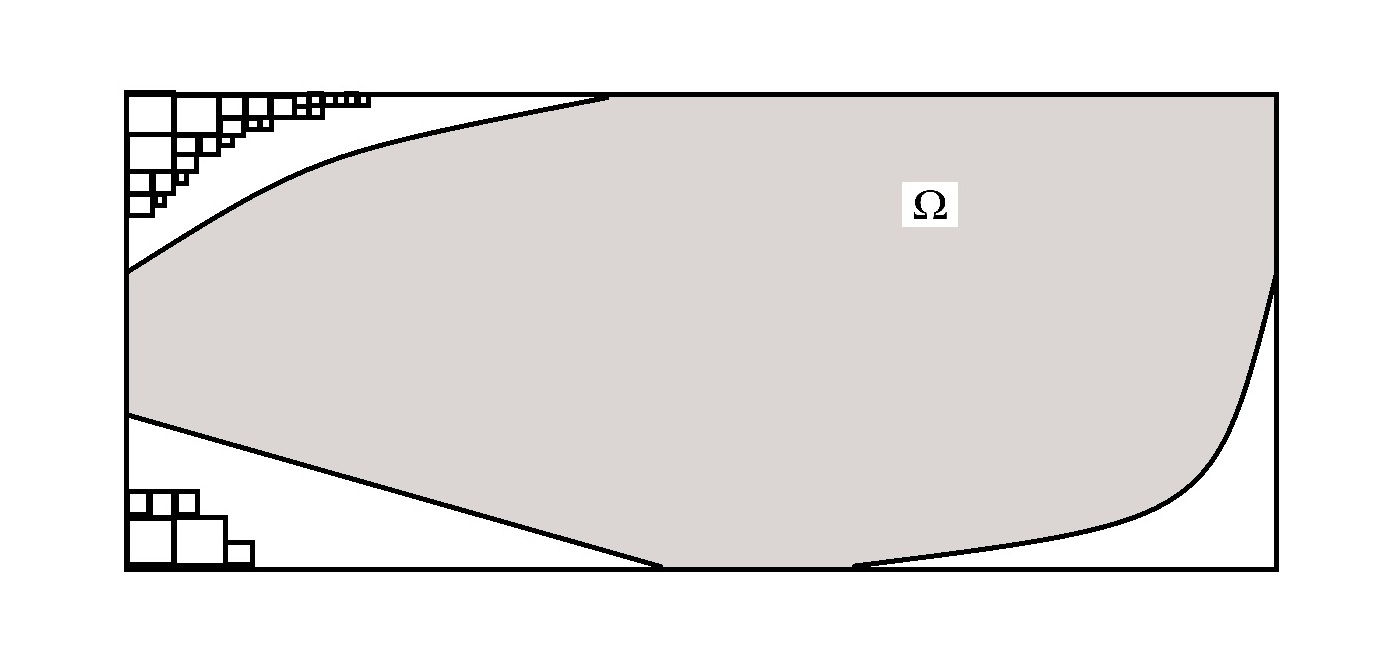}
\caption{Whitney decomposition for the complementary of a convex set}
\label{fig5}
\end{figure}
The convexity of $\mathscr{R}$ and $\Omega$ would imply that
$$\left|\{x\in\mathscr{R}\setminus\overline{\Omega}:\,\mathrm{dist}(x,\partial\Omega)<\epsilon\}\right| \le
|\partial\mathscr{R}\setminus\partial\Omega|\epsilon.$$
Applying this estimate to {\bf Step 2} above gives that
\begin{align*}
k&\le \frac{|\Omega|\omega(n)}{(2\pi)^n}\lambda_k(\Omega)^{n/2}\\
&\quad+  \frac{\omega(n)}{(2\pi)^n}\lambda_k(\Omega)^{\frac{n-1}{2}} \left(|\partial\mathscr{R}\setminus\partial\Omega|\left(\frac{\pi r_{in}(\mathscr{R}\setminus \overline{\Omega})}{|I_{min}|}+5n^{2}\pi\log_2\frac{10|I_{min}|\sqrt{\lambda_k(\Omega)}}{\pi} \right)-C_1(n-1)|\mathscr{R}_{n-1}|\right).
\end{align*}
The proof is complete.
\end{proof}

We now prove the $\epsilon$-loss version of P\'olya's conjecture for large eigenvalues.
\begin{proof}[Proof of Theorem \ref{epsilon-loss-lip}]
Let $\Omega_1=\mathscr{R}\setminus\overline{\Omega}$. By Theorem \ref{weyl-remainder-uniform}, it holds for each $k\in\cn$ that
\begin{align*}
k&\le \frac{|\Omega|\omega(n)}{(2\pi)^n}\lambda_k(\Omega)^{n/2}\\
&\quad+  \frac{\omega(n)}{(2\pi)^n}\lambda_k(\Omega)^{\frac{n-1}{2}} \left(C_{Lip}(\Omega)|\partial (\mathscr{R}\setminus \overline{\Omega})| \left(\frac{\pi r_{in}(\mathscr{R}\setminus \overline{\Omega})}{|I_{min}|}+5n^{2}\pi\log_2\frac{10|I_{min}|\sqrt{\lambda_k(\Omega)}}{\pi} \right)-C_1(n-1)|\mathscr{R}_{n-1}|\right)\\
&\le  \frac{|\Omega|\omega(n)}{(2\pi)^n}\lambda_k(\Omega)^{n/2} \left(1+ \frac{C_{Lip}(\Omega)(2n\mathrm{diam}(\Omega)^{n-1}+|\partial\Omega|)}{\sqrt{\lambda_k}|\Omega|} \left(\frac{\pi }{2}+5n^{2}\pi\log_2\frac{10\mathrm{width}(\Omega)\sqrt{\lambda_k(\Omega)}}{\pi} \right)  \right).
\end{align*}
Taking derivative of the  function
$$f(\Lambda):= \frac{2n\mathrm{diam}(\Omega)^{n-1}+|\partial\Omega|}{\sqrt{\Lambda}|\Omega|} \left(\frac{\pi }{2}+5n^{2}\pi\log_2\frac{10\mathrm{width}(\Omega)\sqrt{\Lambda}}{\pi} \right), $$
we see that
\begin{align*}
f'(\Lambda)&=  -\frac{2n\mathrm{diam}(\Omega)^{n-1}+|\partial\Omega|}{2\Lambda^{3/2}|\Omega|} \left(\frac{\pi }{2}+5n^{2}\pi\log_2\frac{10\mathrm{width}(\Omega)\sqrt{\Lambda}}{\pi} \right)\\
&\quad\quad + \frac{2n\mathrm{diam}(\Omega)^{n-1}+|\partial\Omega|}{\sqrt{\Lambda}|\Omega|} \frac{5n^{2}\pi  \log_2e  }{2\Lambda}.
\end{align*}
We have for each $\Lambda\ge \frac{1}{\mathrm{width}(\Omega)^2}$ that
$$\frac{10\mathrm{width}(\Omega)\sqrt{\Lambda}}{\pi}\ge \frac{ 10}{\pi}>e,$$
and therefore $f(\Lambda)$ is monotone decreasing on $[\frac{1}{\mathrm{width}(\Omega)^2},\infty)$.

Note that for  $\Lambda=\frac{1}{\mathrm{width}(\Omega)^2}$,
$$f(\Lambda)=\frac{(2n \mathrm{diam}(\Omega)^{n-1}+|\partial\Omega|)\mathrm{width}(\Omega)}{|\Omega|} \left(\frac{\pi}{2}+5n^{2}\pi\log_2\frac{10}{\pi} \right)>2n,$$
$$\lim_{\Lambda\to\infty}f(\Lambda)=0,$$
and $C_{Lip}(\Omega)\ge 1$.
So for any $\epsilon\in (0,1)$, we can let $\Lambda> \frac{1}{\mathrm{width}(\Omega)^2}$ be such that
$$ C_{Lip}(\Omega)\frac{2n \mathrm{diam}(\Omega)^{n-1}+|\partial\Omega|}{|\Omega|\sqrt\Lambda} \left(\frac{\pi}{2}+5n^{2}\pi\log_2\frac{10\mathrm{width}(\Omega)\sqrt{\Lambda}}{\pi} \right)= \epsilon,$$
and deduce for any $\lambda_k(\Omega)\ge \Lambda$ that
\begin{align*}
k&\le (1+\epsilon)\frac{|\Omega|\omega(n)}{(2\pi)^n}\lambda_k(\Omega)^{n/2}.
\end{align*}

If $\Omega$ is convex, then applying Corollary \ref{cor-uniform-weyl} gives that
\begin{align*}
k&\le \frac{|\Omega|\omega(n)}{(2\pi)^n}\lambda_k(\Omega)^{n/2} \left(1+ \frac{2n\mathrm{diam}(\Omega)^{n-1}}{\sqrt{\lambda_k}|\Omega|}
\left(\frac{\pi}{2}-\frac{C_1(n-1)}{2n}+5n^{2}\pi\log_2\frac{10\mathrm{width}(\Omega)\sqrt{\lambda_k}}{\pi} \right)  \right).
\end{align*}
The remaining proof is the same. The proof is complete.
\end{proof}

\subsection*{Acknowledgments}
\addcontentsline{toc}{section}{Acknowledgments} \hskip\parindent
We wish to thank Prof. Rupert Frank and Prof. Simon Larson for bringing references \cite{FLW23,FL24b,Larson17} to us and several useful comments,
and thank Dr. Xiang He for bringing the reference \cite{NS05}.
R. Jiang wishes to thank Prof. Xiao Zhong for explaining to him  P\'olya's conjecture
and helpful discussions at University of Jyv\"askyl\"a around 2010.
R. Jiang was partially supported by NNSF of China (12526205 \& 12471094),  F.H. Lin
was in part supported by the National Science Foundation Grant  DMS2247773 and DMS1955249.

\noindent Renjin Jiang \\
\noindent Academy for Multidisciplinary Studies\\
\noindent Capital Normal University\\
\noindent Beijing 100048\\
\noindent {rejiang@cnu.edu.cn}

\

\noindent Fanghua Lin\\
\noindent Courant Institute of Mathematical Sciences\\
\noindent 251 Mercer Street, New York, NY 10012, USA \\
\noindent linf@cims.nyu.edu

\end{document}